\pdfoutput=1
\documentclass[12pt,letter]{amsart}
\usepackage{xcolor}
\usepackage[all,color]{xy}
\usepackage{amsmath,amssymb,amsthm,amscd,amsxtra,amsfonts,mathrsfs,graphicx,enumerate,bm,slashed,array,setspace,amsfonts}
\input xy
\xyoption{all}
\newtheorem{Theorem}{Theorem}[section]
\newtheorem{Lemma}[Theorem]{Lemma}
\newtheorem{Proposition}[Theorem]{Proposition}
\newtheorem{Corollary}[Theorem]{Corollary}

\newtheorem{Remark}[Theorem]{Remark}

\newtheorem{Definition}[Theorem]{Definition}

\newtheorem*{Theorem A}{Theorem A}
\newtheorem*{Corollary B}{Corollary B}
\newtheorem*{Theorem C}{Theorem C}

\newcommand*{\overbar}[1]{\mkern 1.5mu\overline{\mkern-1.5mu#1\mkern-1.5mu}\mkern 1.5mu}
\advance\evensidemargin-.5in
\advance\oddsidemargin-.5in
\advance\textwidth1in

\begin{document}
\author{Charlie Beil}
 \address{Institut f\"ur Mathematik und Wissenschaftliches Rechnen, Universit\"at Graz, Heinrichstrasse 36, 8010 Graz, Austria.}
 \email{charles.beil@uni-graz.at}
\title[Nonnoetherian singularities and their noncommutative blowups]{Nonnoetherian singularities and their noncommutative blowups}
 \date{}

\begin{abstract}
We establish a new fundamental class of varieties in nonnoetherian algebraic geometry related to the central geometry of dimer algebras.
Specifically, given an affine algebraic variety $X$ and a finite collection of non-intersecting positive dimensional algebraic sets $Y_i \subset X$, we construct a nonnoetherian coordinate ring whose variety coincides with $X$ except that each $Y_i$ is identified as a distinct positive dimensional closed point.
We then show that the noncommutative blowup of such a singularity is a noncommutative desingularization, in a suitable geometric sense. 
\end{abstract}

\maketitle

\tableofcontents

\section{Introduction}

The primary objectives of this article are (i) to extend the framework of depictions, introduced in \cite{B4}, to a much larger class of varieties with nonnoetherian coordinate rings; and (ii) to show that noncommutative blowups of these varieties are noncommutative desingularizations, in a suitable sense. 
This framework was originally developed to provide the geometric tools needed to understand the representation theory of a class of quiver algebras called non-cancellative dimer algebras (e.g., \cite{B2,B3,B5}).
Dimer algebras arose in string theory \cite{HK,FHMSVW}, and have found wide application to many areas of mathematics (e.g., \cite{BKM, Br, FHKV, He, IN, IU, MR}). 
Depictions have enabled various notions in noncommutative algebraic geometry--such as noncommutative crepant resolutions \cite{V}, homological homogeneity \cite{BH}, and Azumaya loci--to be generalized to tiled matrix algebras that are not finitely generated modules over their centers \cite{B2,B3}; we will consider some of these generalizations here.
The underlying ideas of nonnoetherian algebraic geometry also suggest possible directions towards a new theory of quantum gravity \cite{B1,B6}.

Throughout, let $k$ be an algebraically closed field, and let $R$ be a subalgebra of an affine coordinate ring $S$ over $k$.
It is generally believed that nonnoetherian algebras do not admit concrete geometric descriptions.  
For example, consider the subalgebras of the polynomial rings $S_1 = k[x,y]$ and $S_2 = k[x,y,z]$,
\begin{align*}
R_1 & = k[x] + x(x-1)(x-2)S_1,\\
R_2 & = k[x^2 - y - z^2] + (x^2 - y, z-5)(x-z,y)S_2.
\end{align*}
We may ask, informally, what their maximal spectra $\operatorname{Max}R$ `look like', but such a question initially appears hopeless, at least in terms of geometries we can visualize.

We could instead consider the simpler subalgebras
\begin{align*}
R_1' & = k + x(x-1)(x-2)S_1,\\
R_2' & = k + (x^2 - y, z-5)(x-z,y)S_2.
\end{align*}
Both are of the form $R = k + I$, where $I$ is an ideal of $S$.
A geometric description of such subalgebras was introduced in \cite{B4}: the maximal spectrum $\operatorname{Max}R$ of $R$ coincides with the algebraic variety $\operatorname{Max}S$, except that the zero locus $\mathcal{Z}(I) \subset \operatorname{Max}S$  is identified as a single `smeared-out' point. 

In particular, we may view the variety $\operatorname{Max}R_1'$ as $\mathbb{A}_k^2$, with the union of the three lines
\begin{equation} \label{3 lines}
\mathcal{Z}(x) = \{ x = 0 \}, \ \ \ \ \mathcal{Z}(x-1) = \{ x = 1\}, \ \ \ \ \mathcal{Z}(x-2) = \{ x = 2\},
\end{equation}
identified as a single 1-dimensional point.
Similarly, we may view the variety $\operatorname{Max}R_2'$ as $\mathbb{A}_k^3$, with the union of the two curves
\begin{equation} \label{two curves}
\mathcal{Z}(x^2 - y, z-5) \ \ \ \ \text{ and } \ \ \ \ \mathcal{Z}(x-z,y)
\end{equation}
identified as a single 1-dimensional point.

These geometric pictures are made precise using depictions and geometric dimension.
A \textit{depiction} of a nonnoetherian domain $R$ is a finitely generated $k$-algebra $S$ that is as close to $R$ as possible, in a suitable geometric sense (Definition \ref{def dep}).
In particular, if $R$ is depicted by $S$, then $R$ and $S$ have equal Krull dimension, and their maximal spectra are birationally equivalent \cite[Theorem 2.5]{B4}.
Furthermore, the locus where $R$ and $S$ locally coincide,
\begin{equation*} \label{wheel}
U_{S/R} := \left\{ \mathfrak{n} \in \operatorname{Max}S \ | \ R_{\mathfrak{n} \cap R} = S_{\mathfrak{n}} \right\},
\end{equation*}
is open dense in $\operatorname{Max}S$ \cite[Proposition 2.4]{B4}. 

Algebras of the form $R = k + I$, with $\operatorname{dim}S/I \geq 1$, comprise an elementary class of examples in nonnoetherian algebraic geometry.
Two ideals $I_1, I_2 \subset S$ are said to be coprime if $I_1 + I_2 = S$; equivalently, their zero loci in $\operatorname{Max}S$ do not intersect, 
$$\mathcal{Z}(I_1) \cap \mathcal{Z}(I_2) = \emptyset.$$
In this article, we consider the question: \textit{given a collection of pairwise coprime ideals
$$I_1, \ldots, I_n \subset S,$$
is there a nonnoetherian ring $R$ for which $\operatorname{Max}R$ coincides with $\operatorname{Max}S$, except that each $\mathcal{Z}(I_i)$ is identified as a distinct closed point of $\operatorname{Max}R$?}
We will show that this question has a positive answer, with $R$ given by the intersection
$$R = \cap_i \left( k + I_i \right).$$

Our first main theorem is the following.

\begin{Theorem A} \label{main} (Propositions \ref{star}, \ref{dim SI2}, and Theorem \ref{main'}.)
Let $X$ be an affine algebraic variety over $k$ with coordinate ring $S$.
Consider a collection of pairwise non-intersecting algebraic sets $Y_1, \ldots, Y_n$ of $X$, where each ideal $I(Y_i)$ is proper, nonzero, and non-maximal.
Then the maximal spectrum of the ring
\begin{equation} \label{wait}
R := \cap_i (k + I(Y_i))
\end{equation}
coincides with $X$ except that each $Y_i$ is identified as a distinct closed point.
In particular, the locus $U_{S/R} \subset X$ is given by the intersection of the complements $Y_i^c$,
$$U_{S/R} = \cap_i Y_i^c.$$
Furthermore, we have:
\begin{itemize}
 \item[(i)] $R$ is nonnoetherian if and only if there is some $i$ for which $\operatorname{dim}Y_i \geq 1$.
 \item[(ii)] $R$ is depicted by $S$ if and only if for each $i$, $\operatorname{dim}Y_i \geq 1$.
\end{itemize}
\end{Theorem A}

Theorem A answers our initial question in a surprisingly simple way: observe that the subalgebras $R_1$ and $R_2$ are of the form (\ref{wait}):
\begin{align*}
R_1 & = k[x] + x(x-1)(x-2)S_1 \\
& = (k + xS_1) \cap (k + (x-1)S_1) \cap (k + (x-2)S_1),
\end{align*}
and 
\begin{align*}
R_2 & = k[x^2 - y - z^2] + (x^2 - y, z-5)(x-z,y)S_2 \\
& = (k + (x^2 - y, z - 5)S_2) \cap (k + (x - z, y)S_2).
\end{align*}
The variety $\operatorname{Max}R_1$ therefore looks exactly like $\mathbb{A}^2_k$, except that each of the three lines in (\ref{3 lines}) is identified as a distinct $1$-dimensional point.
Similarly, $\operatorname{Max}R_2$ looks exactly like $\mathbb{A}_k^3$, except that each of the curves in (\ref{two curves}) is identified as a distinct $1$-dimensional point.

To note, it is peculiar that by adjoining to $R'_2$ the polynomial $x^2 - y - z^2$,
$$R_2 = R'_2[x^2 - y - z^2],$$
the single $1$-dimensional point of $\operatorname{Max}R'_2$ separates into two distinct $1$-dimensional points, while all other points of $\operatorname{Max}R'_2$ are left unchanged.

Theorem A also implies the following generalization of the fact that, given any maximal ideal $\mathfrak{n}$ of $S$, $S$ decomposes as the sum $S = k + \mathfrak{n}$.

\begin{Corollary B} \label{Earth}
Let $I$ be a proper non-maximal nonzero radical ideal of $S$, and set $R = k + I$.
The following are equivalent:
\begin{itemize}
 \item[(i)] $\operatorname{dim}S/I \geq 1$.
 \item[(ii)] $R$ is nonnoetherian.
 \item[(iii)] $R$ is depicted by $S$.
\end{itemize}
\end{Corollary B}

In particular, $R = k + I$ is noetherian if and only if $\operatorname{dim}S/I = 0$, that is,
$$I = \mathfrak{n}_1 \cap \cdots \cap \mathfrak{n}_{\ell}$$ 
for some maximal ideals $\mathfrak{n}_1, \ldots, \mathfrak{n}_{\ell} \in \operatorname{Max}S$.
The implication (ii) $\Rightarrow$ (i) was also shown by Stafford in \cite[Lemma 1.4]{St} using different methods.

In Section \ref{sheaves}, we define a sheaf of depictions on an affine scheme $X$ to be a sheaf of algebras that is a depiction on each principal open set of $X$.
We show that the sheafification of a depiction $S$ of $R$ is a sheaf of depictions on $\operatorname{Spec}R$.

In Section \ref{nbons}, we consider nonnoetherian coordinate rings in the setting of noncommutative algebraic geometry.
Let $S$ be a finite type normal integral domain, let $Y_1, \ldots, Y_n$ be positive dimensional proper subvarieties of $\operatorname{Max}S$ that intersect the smooth locus, and denote by $I_i := I(Y_i)$ their radical ideals in $S$.
By Theorem A, $R := \cap_i (k + I_i)$ is a nonnoetherian coordinate ring with $n$ positive dimensional closed points,
$$\mathfrak{m}_i := I_i \cap R \in \operatorname{Spec}R.$$
Following \cite[Section R]{L}, we call the endomorphism ring
\begin{equation} \label{cup}
A := \operatorname{End}_R( _RR \oplus \bigoplus_i \mathfrak{m}_i )
\end{equation}
the `noncommutative blowup' of $\operatorname{Max}R$ at the points $\mathfrak{m}_1, \ldots, \mathfrak{m}_n$.
We would like to know whether $A$ is a desingularization of its center $R$.

A resolution of a singularity $X$ is a proper birational morphism of schemes $Y \to X$ such that $Y$ is smooth. 
If we omit the requirement of properness, then we may say that $Y \to X$ is a desingularization of $X$.
We note the following:
\begin{enumerate}
 \item[(a)] Birationality implies that $X$ and $Y$ have isomorphic function fields,
$$\operatorname{Frac}k[X] \cong \operatorname{Frac}k[Y].$$
 \item[(b)] Let $\operatorname{Spec}S$ be an affine open subset of $Y$.
Then $\operatorname{Spec}S$ is smooth over $\operatorname{Spec}k$ at a closed point $\mathfrak{n} \in \operatorname{Spec}S$ if and only if\footnote{Since we are assuming $k$ algebraically closed, $\operatorname{Spec}S$ is smooth at $\mathfrak{n}$ if and only if $S_{\mathfrak{n}}$ is regular \cite[III.10.0.3]{H}.} the global dimension of $S_{\mathfrak{n}}$, the projective dimension of the residue field $S_{\mathfrak{n}}/\mathfrak{n} \cong k$, and the Krull dimension of $S_{\mathfrak{n}}$ all coincide \cite{AB1,AB2,S},
$$\operatorname{gldim}S_{\mathfrak{n}} = \operatorname{pd}_{S_{\mathfrak{n}}}(S_{\mathfrak{n}}/\mathfrak{n}) = \operatorname{dim}S_{\mathfrak{n}}.$$
\end{enumerate}

Following Brown and Hajarnavis's notion of a homologically homogeneous ring \cite{BH}, and Van den Bergh's notion of a noncommutative crepant resolution \cite{V}, we say that a noncommutative algebra $A$, module-finite over its noetherian center $R$, is a noncommutative desingularization of $R$ if the following two conditions hold:
\begin{enumerate}
 \item[(a$'$)] $\operatorname{Frac}R$ and $A \otimes_R \operatorname{Frac}R$ are Morita equivalent.
 \item[(b$'$)] For each closed point $\mathfrak{m} \in \operatorname{Spec}R$, the central localization $A_{\mathfrak{m}} := A \otimes_R R_{\mathfrak{m}}$ satisfies
$$\operatorname{gldim}A_{\mathfrak{m}} = \operatorname{pd}_{A_{\mathfrak{m}}}(A_{\mathfrak{m}}/\mathfrak{m}) = \operatorname{dim}R_{\mathfrak{m}}.$$
\end{enumerate}

However, the singularities we are considering here are nonnoetherian, and their noncommutative blowups are not module-finite over their centers (just as the case for non-cancellative dimer algebras).
Condition (b$'$) must therefore be modified to allow for this generality.
Such a modification is possible for tiled matrix algebras using the notions of `cycle algebra' and `cyclic localization', introduced in \cite{B3,B4} (Definition \ref{cycleregulardef}).
In cases of interest, if the center $R$ is noetherian, then the cycle algebra and center coincide, and cyclic localization is the same as central localization \cite[Theorem 4.1]{B4}.
We thus replace (b$'$) with the following condition:
\begin{enumerate}
 \item[(b$''$)] Let $S$ be the cycle algebra of $A$.
For each closed point $\mathfrak{m} \in \operatorname{Spec}R$ and each minimal prime $\mathfrak{q} \in \operatorname{Spec}S$ over $\mathfrak{m}$, the cyclic localization $A_{\mathfrak{q}}$ satisfies
$$\operatorname{gldim}A_{\mathfrak{q}} = \operatorname{pd}_{A_{\mathfrak{q}}}(A_{\mathfrak{q}}/\mathfrak{q}) = \operatorname{dim}S_{\mathfrak{q}}.$$
\end{enumerate}

Our second main theorem is the following.

\begin{Theorem C} (Theorem \ref{akzily}.)
Let $A$ be the endomorphism ring in (\ref{cup}), and let $S$ be its cycle algebra.
If each $Y_i$ is irreducible, or $n = 1$, then $A$ is a noncommutative desingularation of its center $R$:
\begin{itemize}
 \item[-] $\operatorname{Frac}R$ and $A \otimes_R \operatorname{Frac}R$ are Morita equivalent, and
 \item[-] for each $i \in [1,n]$ and minimal prime $\mathfrak{q} \in \operatorname{Spec}S$ over $\mathfrak{m}_i$, we have
$$\operatorname{gldim}A_{\mathfrak{q}} = \operatorname{pd}_{A_{\mathfrak{q}}}(A_{\mathfrak{q}}/\mathfrak{q}) = \operatorname{dim}S_{\mathfrak{q}}.$$
\end{itemize} 
Furthermore, the Azumaya locus of $A$ and the noetherian locus $U_{S/R}$ of $R$ coincide.
\end{Theorem C}

\section{Preliminary definitions}

Given an integral domain $k$-algebra $S$, denote by $\operatorname{Max}S$, $\operatorname{Spec}S$, $\operatorname{Frac}S$, and $\operatorname{dim}S$ the maximal spectrum (or variety), prime spectrum (or affine scheme), fraction field, and Krull dimension of $S$ respectively.
For a subset $I \subset S$, set $\mathcal{Z}(I) := \left\{ \mathfrak{n} \in \operatorname{Max}S \ | \ \mathfrak{n} \supseteq I \right\}$.

Given a (not-necessarily-commutative) $k$-algebra $A$ and an $A$-module $V$, denote by $\operatorname{gldim}A$ and $\operatorname{pd}_A(V)$ the left global dimension of $A$ and projective dimension of $V$, respectively.
By module we mean left module, unless stated otherwise.

The following definitions have been instrumental in studying dimer algebras (e.g., \cite{B2,B3,B5}).

\begin{Definition} \label{def dep} \rm{\cite[Definition 3.1]{B4}
Let $S$ be an integral domain and a finitely generated $k$-algebra, and let $R$ be a subalgebra of $S$.
\begin{itemize}
 \item We say $S$ is a \textit{depiction} of $R$ if the morphism
$$\iota_{S/R} : \operatorname{Spec}S \rightarrow \operatorname{Spec}R, \ \ \ \ \mathfrak{q} \mapsto \mathfrak{q} \cap R,$$
is surjective, and
$$U_{S/R} := 
\left\{ \mathfrak{n} \in \operatorname{Max}S \ | \ R_{\mathfrak{n} \cap R} = S_{\mathfrak{n}} \right\} = 
\left\{ \mathfrak{n} \in \operatorname{Max}S \ | \ R_{\mathfrak{n} \cap R} \text{ is noetherian} \right\} \not = \emptyset.$$
 \item The \textit{geometric height} of $\mathfrak{p} \in \operatorname{Spec}R$ is the minimum
$$\operatorname{ght}(\mathfrak{p}) := \operatorname{min} \left\{ \operatorname{ht}_S(\mathfrak{q}) \ | \ \mathfrak{q} \in \iota^{-1}_{S/R}(\mathfrak{p}), \ S \text{ a depiction of } R \right\}.$$
The \textit{geometric dimension} of $\mathfrak{p}$ is\footnote{Recall that if $S$ is an integral domain and a finitely generated $k$-algebra, then for each $\mathfrak{q} \in \operatorname{Spec}S$, we have $\operatorname{dim}S/\mathfrak{q} = \operatorname{dim}S - \operatorname{ht}(\mathfrak{q})$.}
$$\operatorname{gdim} \mathfrak{p} := \operatorname{dim}R - \operatorname{ght}(\mathfrak{p}).$$
 \end{itemize}
} \end{Definition}

For brevity, we will often write $\iota$ for $\iota_{S/R}$.
To note, if $R$ is depicted by $S$, then $R$ is noetherian if and only if $R = S$ \cite[Theorem 3.12]{B4}.

Now let $B$ be an integral domain and $k$-algebra, and let 
$$A = [A^{ij}] \subset M_n(B)$$ 
be a tiled matrix ring, that is, each diagonal $A^i := A^{ii}$ is a unital subalgebra of $B$.
The following definitions, with the exception of residue module, were introduced in \cite{B3}; the notion of residue module we are considering here is new. 

\begin{Definition} \label{cycleregulardef} \rm{\cite[Definition 3.1]{B3}
Set
$$R := k[\cap_i A^i] \ \ \ \text{ and } \ \ \ S := k[\cup_i A^i].$$
We call $S$ the \textit{cycle algebra} of $A$, and in cases of interest, $R$ is the center of $A$ \cite[Theorem 4.1]{B4}.
The \textit{cyclic localization} of $A$ at a prime $\mathfrak{q} \in \operatorname{Spec}S$ is the algebra
$$A_{\mathfrak{q}} := \left\langle \left[ \begin{matrix} A^1_{\mathfrak{q} \cap A^1} & A^{12} & \cdots & A^{1n} \\
A^{21} & A^2_{\mathfrak{q} \cap A^2} & & A^{2n} \\
\vdots & & \ddots & \vdots\\
A^{n1} & A^{n2} & \cdots & A^n_{\mathfrak{q} \cap A^n} \end{matrix} \right] \right\rangle
\subset M_n(\operatorname{Frac}B).$$
The \textit{residue module} $A_{\mathfrak{q}}/\mathfrak{q}$ of $A$ at $\mathfrak{q}$ is the quotient of $A_{\mathfrak{q}}$ by the ideal
$$A_{\mathfrak{q}} \left[ \begin{matrix} \mathfrak{q} \cap A^1 & 0 & \cdots & 0 \\ 
0 & \mathfrak{q} \cap A^2 & & 0 \\
\vdots & & \ddots & \vdots \\
0 & 0 & \cdots & \mathfrak{q} \cap A^n \end{matrix} \right] A_{\mathfrak{q}}.$$
}\end{Definition}

\begin{Remark} \label{void} \rm{
If $R = S$, that is, $A^i = A^j$ for each $i,j$, then cyclic localization coincides with the usual notion of central localization:
$$A_{\mathfrak{q}} \cong A \otimes_R R_{\mathfrak{q}} \ \ \ \text{ and } \ \ \ A_{\mathfrak{q}}/\mathfrak{q} \cong A \otimes_R R_{\mathfrak{q}}/\mathfrak{q}.$$
}\end{Remark}

\begin{Definition} \label{def2nc} \rm{\cite[Definition 3.2]{B3}
We say $A$ is \textit{cycle regular} at $\mathfrak{m} \in \operatorname{Max}R$ if for each minimal prime $\mathfrak{q} \in \operatorname{Spec}S$ over $\mathfrak{m}$, we have\footnote{In \cite{B3}, we defined $A$ to be cycle regular at $\mathfrak{m} \in \operatorname{Max}R$ if, for each minimal prime $\mathfrak{q} \in \operatorname{Spec}S$ over $\mathfrak{m}$ and each simple $A_{\mathfrak{q}}$-module $V$, we have $\operatorname{gldim}(A_{\mathfrak{q}}) = \operatorname{pd}_{A_{\mathfrak{q}}}(V) = \operatorname{dim}S_{\mathfrak{q}}$.
In this article, we replace the set of simple $A_{\mathfrak{q}}$-modules $V$ with the residue module $A_{\mathfrak{q}}/\mathfrak{q}$, which, in our case, is a direct sum of all such simples (see Propositions \ref{simple prop} and \ref{localized}).}
$$\operatorname{gldim}(A_{\mathfrak{q}}) = \operatorname{pd}_{A_{\mathfrak{q}}}(A_{\mathfrak{q}}/\mathfrak{q}) = \operatorname{dim}S_{\mathfrak{q}}.$$
If, in addition, $\operatorname{Frac}R$ and $A \otimes_R \operatorname{Frac}R$ are Morita equivalent, then we say $A$ is a \textit{noncommutative desingularization} of $R$.
}\end{Definition}

\section{Nonnoetherian coordinate rings with multiple positive dimensional points} \label{firstsect}

Let $S$ be an integral domain and a finitely generated $k$-algebra.
Let $I_1, \ldots, I_n$ be a collection of proper non-maximal nonzero radical ideals of $S$ such that, for each $i \not = j$, $\mathcal{Z}(I_i) \cap \mathcal{Z}(I_j) = \emptyset$; equivalently, $I_i$ and $I_j$ are coprime: $I_i + I_j = S$.
Unless stated otherwise, we denote by $R$ the algebra
$$R := \cap_i (k + I_i).$$

\begin{Remark} \rm{
If some $I_j$ were a maximal ideal of $S$, then $k + I_j = S$, whence $R = \cap_{i \not = j} (k + I_i)$.
The assumption that each $I_i$ is proper and nonzero implies that $\operatorname{dim}S \geq 1$.
}\end{Remark}

\begin{Lemma} \label{Andy}
Suppose $n \geq 2$.
For each $i \in [1,n]$, there are elements $a,b \in R$ satisfying
$$a \in I_i \setminus \left( \cup_{j \not = i} I_j \right), \ \ \ \ b \in \left( \cap_{j \not = i} I_j  \right) \setminus I_i,$$
and which sum to unity, $a + b = 1$.
\end{Lemma}

\begin{proof}
Fix $i \in [1,n]$.
By assumption, we have
$$\mathcal{Z}(1) = \emptyset = \cup_{j \not = i} \left( \mathcal{Z}(I_i) \cap \mathcal{Z}(I_j) \right) = \mathcal{Z}(I_i) \cap \left( \displaystyle \cup_{j \not = i} \mathcal{Z}(I_j) \right) = \mathcal{Z}(I_i + \cap_{j \not = i} I_j).$$
Whence
$$1 \in I_i + \cap_{j \not = i}I_j.$$
Thus there is some $a \in I_i$ and $b \in \cap_{j \not = i} I_j$ such that $a + b = 1$.
In particular,
$$a = 1 - b \in I_i \cap \left( \cap_{j \not = i} \left( k + I_j \right) \right) \subset R.$$
It also follows that for each $j \not = i$,
$$a = 1 - b \in I_i \setminus I_j \ \ \ \text{ and } \ \ \ b = 1 - a \in I_j \setminus I_i.$$
\end{proof}

\begin{Proposition} \label{star}
Each ideal $I_i \cap R$ is a distinct closed point of $\operatorname{Spec}R$.
\end{Proposition}

\begin{proof}
Fix $i$.
For each $a \in R \subseteq (k + I_i)$, there is some $\alpha_i \in k$ and $b_i \in I_i$ such that $a = \alpha_i + b_i$.
In particular, there is an algebra epimorphism
$$R \to k, \ \ \ a \mapsto \alpha_i,$$
with kernel $I_i \cap R$; whence an algebra isomorphism $R/(I_i \cap R) \cong k$.
Furthermore, there exists some $a \in (I_i \cap R) \setminus (\cup_{j \not = i} I_j)$, by Lemma \ref{Andy}.
Thus, for each $j \not = i$,
$$I_j \cap R \not = I_i \cap R.$$
Therefore each $I_i \cap R$ is a distinct maximal ideal of $R$.
\end{proof}

\begin{Proposition} \label{dim SI2}
The locus $U_{S/R} := \left\{ \mathfrak{n} \in \operatorname{Max}S \ | \ R_{\mathfrak{n} \cap R} = S_{\mathfrak{n}}\right\}$ is given by
\begin{equation*}
U_{S/R} = \left( \cup_i \mathcal{Z}(I_i) \right)^c.
\end{equation*}
\end{Proposition}

\begin{proof}
(i) We first claim that $U_{S/R} \subseteq \left( \cup_i \mathcal{Z}(I_i) \right)^c$.
Indeed, let $\mathfrak{n} \in \cup_i \mathcal{Z}(I_i)$.
Then $\mathfrak{n}$ contains some $I_i$.
By assumption, $I_i$ is a non-maximal radical ideal of $S$.
Thus there is another maximal ideal $\mathfrak{n}' \not = \mathfrak{n}$ of $S$ which contains $I_i$.
Whence
$$I_i \cap R \subseteq \mathfrak{n} \cap R \not = R \ \ \ \text{ and } \ \ \ I_i \cap R \subseteq \mathfrak{n}' \cap R \not = R.$$
But $I_i \cap R$ is a maximal ideal of $R$ by Proposition \ref{star}.
Therefore
$$\mathfrak{n} \cap R = I_i \cap R = \mathfrak{n}' \cap R.$$

Now fix $c \in \mathfrak{n} \setminus \mathfrak{n}'$.
Assume to the contrary that $c \in R_{\mathfrak{n} \cap R}$.
Then there is some $a \in R$ and $b \in R \setminus (\mathfrak{n} \cap R)$ such that $c = \frac ab$.
Whence
$$a = bc \in \mathfrak{n} \cap R = \mathfrak{n}' \cap R.$$
In particular, $bc \in \mathfrak{n}'$ with $b,c \in S$.
Therefore
\begin{equation} \label{spaceship}
b \in \mathfrak{n}',
\end{equation}
since $c \not \in \mathfrak{n}'$ and $\mathfrak{n}'$ is a prime ideal of $S$.
But $b \in R$ and
$$b \not \in \mathfrak{n} \cap R = \mathfrak{n}' \cap R.$$
Whence $b \not \in \mathfrak{n}'$, a contradiction to (\ref{spaceship}).
Thus $c \in S_{\mathfrak{n}} \setminus R_{\mathfrak{n} \cap R}$.
Therefore $\mathfrak{n} \in U_{S/R}^c$.

(ii) We now claim that $U_{S/R} \supseteq \left( \cup_i \mathcal{Z}(I_i) \right)^c$.\footnote{This claim was proven in the special case $n = 1$ in \cite[Proposition 2.8]{B4}.}
Let $\mathfrak{n} \in \left( \cup_i \mathcal{Z}(I_i) \right)^c$.
For each $i$, $\mathfrak{n} \not \supseteq I_i$.
In particular, for each $i$ there is some $c_i \in I_i \setminus \mathfrak{n}$.
Furthermore, since $\mathfrak{n}$ is prime, we have
\begin{equation} \label{yoohoo}
c := c_1 \cdots c_n \in \left( \cap_i I_i \right) \setminus \mathfrak{n}.
\end{equation}

Now let $\frac ab \in S_{\mathfrak{n}}$, with $a \in S$ and $b \in S \setminus \mathfrak{n}$.
Then by (\ref{yoohoo}),
$$ac \in R \ \ \ \text{ and } \ \ \ bc \in R \setminus (\mathfrak{n} \cap R).$$
Thus
$$\frac ab = \frac{ac}{bc} \in R_{\mathfrak{n} \cap R}.$$
Whence
$$S_{\mathfrak{n}} \subseteq R_{\mathfrak{n} \cap R} \subseteq S_{\mathfrak{n}}.$$
Therefore $R_{\mathfrak{n} \cap R} = S_{\mathfrak{n}}$.
\end{proof}

\begin{Lemma} \label{today}
If $J$ is a proper ideal of $R$ and $\mathcal{Z}(J) \cap U_{S/R} = \emptyset$, then $J$ is contained in some $I_i$.
\end{Lemma}

\begin{proof}
Suppose the hypotheses hold, and let $\mathfrak{n} \in \mathcal{Z}(J)$.
Then $\mathfrak{n} \in U^c_{S/R}$.
Whence $\mathfrak{n} \in \cup_j \mathcal{Z}(I_j)$ by Proposition \ref{dim SI2}.
Thus $\mathfrak{n}$ contains some $I_i$.
Consequently,
$$I_i \cap R \subseteq \mathfrak{n} \cap R \not = R.$$
Whence $I_i \cap R = \mathfrak{n} \cap R$ since $I_i \cap R \in \operatorname{Max}R$ by Proposition \ref{star}.
Therefore
$$J = J \cap R \subseteq \mathfrak{n} \cap R = I_i \cap R \subseteq I_i.$$
\end{proof}

\begin{Lemma} \label{external}
For each $i$,
\begin{equation} \label{R = }
R_{I_i \cap R} = (k + I_i)_{I_i}.
\end{equation}
\end{Lemma}

\begin{proof}
The lemma is trivial if $n = 1$, so suppose $n \geq 2$.
Fix $i \in [1,n]$.
By Lemma \ref{Andy}, there is some
$$c \in \left( \cap_{j \not = i} I_j \right) \cap R \setminus I_i.$$

Let $\frac ab \in \left(k + I_i \right)_{I_i}$, with $a \in k + I_i$ and $b \in \left( k + I_i \right) \setminus I_i$.
Since $c$ is in $R$, $c$ is in $k + I_i$.
Thus, since $a$ is also in $k+ I_i$, the product $ac$ is in $k + I_i$.
Furthermore, since $c$ is in $\cap_{j \not = i} I_j$, $ac$ is in $\cap_{j \not = i}I_j$.
Whence, $ac$ is in $R$.
Similarly, $bc$ is in $R$. 
But $bc$ is not in $I_i$ since $I_i$ is a maximal, hence prime, ideal of $k + I_i$.
Consequently,
$$\frac ab = \frac{ac}{bc} \in R_{I_i \cap R}.$$
It follows that
$$(k + I_i)_{I_i} \subseteq R_{I_i \cap R}.$$

Conversely,
$$R_{I_i \cap R} = \left( \cap_j \left( k + I_j \right) \right)_{I_i \cap R} \subseteq \cap_j  \left( k + I_j \right)_{I_i \cap \left( k + I_j \right)} \subseteq \left( k + I_i \right)_{I_i \cap \left( k + I_i \right)} = \left( k + I_i \right)_{I_i}.$$
Therefore (\ref{R = }) holds.
\end{proof}

For the following, note that if $\mathfrak{n}_1, \ldots, \mathfrak{n}_{\ell}$ are maximal ideals of $S$, then
$$I = \mathfrak{n}_1 \cap \cdots \cap \mathfrak{n}_{\ell} = \sqrt{\mathfrak{n}_1 \cdots \mathfrak{n}_{\ell}}$$
is a radical ideal of $S$ satisfying $\operatorname{dim}S/I = 0$.

\begin{Lemma} \label{R = k + I}
Suppose $I$ is a radical ideal of $S$ satisfying $\operatorname{dim}S/I = 0$.
Then the ring $R = k + I$ is noetherian.
\end{Lemma}

\begin{proof}
Suppose $R$ is nonnoetherian.
We claim that
$$\operatorname{dim}S/I = \operatorname{dim}\mathcal{Z}(I) \stackrel{\textsc{(i)}}{=} \operatorname{dim}U^c_{S/R} \stackrel{\textsc{(ii)}}{\geq} 1.$$
Indeed, (\textsc{i}) holds since by Proposition \ref{dim SI2},
\begin{equation} \label{ZI}
\mathcal{Z}(I) = U^c_{S/R}.
\end{equation}

To show (\textsc{ii}), recall \cite[Theorem 3.13.2]{B4}:\footnote{In the published version of \cite[Theorem 3.13.2]{B4}, $S$ is assumed to be a depiction of $R$, but this is not used in the proof of the theorem.} if $R$ is a nonnoetherian subalgebra of a finitely generated $k$-algebra $S$, and there is some $\mathfrak{m} \in \iota(U^c_{S/R})$ satisfying $\sqrt{\mathfrak{m}S} = \mathfrak{m}$, then
$$\operatorname{dim}U^c_{S/R} \geq 1.$$
In our case, $R = k + I$ is nonnoetherian and $\sqrt{IS} = I$.
Moreover, $I$ is in $\iota(U_{S/R}^c)$: for $\mathfrak{n} \in \mathcal{Z}(I)$, we have
$$I \stackrel{\textsc{(a)}}{=} \mathfrak{n} \cap R  = \iota(\mathfrak{n}) \in \iota(\mathcal{Z}(I)) \stackrel{\textsc{(b)}}{=} \iota(U_{S/R}^c),$$
where (\textsc{a}) holds since $I$ is maximal in $R$, and (\textsc{b}) holds by (\ref{ZI}). 
Therefore (\textsc{ii}) holds.
\end{proof}

\begin{Proposition} \label{internal}
Suppose each $I_i$ is a radical ideal of $S$.
\begin{enumerate}
 \item If $\operatorname{dim}S/I_i = 0$ for each $i$, then $R$ is noetherian.
 \item If $\operatorname{dim}S/I_i = 0$, then the localization $R_{I_i \cap R}$ is noetherian.
 \end{enumerate}
\end{Proposition}

\begin{proof}
(1) Suppose  $\operatorname{dim}S/I_i = 0$ for each $i$.
Set
$$R^m := \cap_{i = 1}^m \left( k + I_i \right).$$
We proceed by induction on $m$.

By Lemma \ref{R = k + I}, $R^1$ is noetherian.
So suppose $R^m$ is noetherian; we claim that $R^{m+1}$ is noetherian.

Indeed, recall that a ring $T$ is noetherian if there is a finite set of elements $a_1, \ldots, a_m \in T$ such that $(a_1, \ldots, a_m)T = T$, and each localization $T_{a_i} := T[a_i^{-1}]$ is noetherian (e.g., \cite[Proposition III.3.2]{H}).

By Lemma \ref{Andy}, $R^{m+1}$ contains elements
\begin{equation} \label{I m+1}
a \in I_{m+1} \setminus \left( \cup_{i = 1}^m I_i \right) \ \ \ \text{ and } \ \ \ b \in \left( \cap_{i = 1}^m I_i \right) \setminus I_{m+1}
\end{equation}
satisfying $a + b = 1$.
In particular,
$$(a,b)R^{m+1} = R^{m+1}.$$
Furthermore, (\ref{I m+1}) implies
\begin{equation} \label{localizations...}
R^{m+1}_a = R^m_a \ \ \ \text{ and } \ \ \ R^{m+1}_b = (k + I_{m+1})_b.
\end{equation}
But $R^m$ is noetherian by assumption, and $(k + I_{m+1})$ is noetherian by Lemma \ref{R = k + I}.
Thus the localizations (\ref{localizations...}) are noetherian.
Therefore $R^{m+1}$ is noetherian, proving our claim.

(2) Now suppose $\operatorname{dim}S/I_i = 0$.
Then the ring $k + I_i$ is noetherian by Lemma \ref{R = k + I}.
Thus the localization $(k + I_i)_{I_i}$ is noetherian.
But $R_{I_i \cap R} = (k + I_i)_{I_i}$ by Lemma \ref{external}.
Therefore $R_{I_i \cap R}$ is noetherian.
\end{proof}

\begin{Proposition} \label{dim SI}
Suppose $I$ is a nonzero radical ideal of $S$ satisfying $\operatorname{dim}S/I \geq 1$.
Then the ring $R = k + I$ is nonnoetherian and $I$ contains a strict infinite ascending chain of ideals of $R$.\footnote{This proposition is erroneously claimed as a corollary to \cite[Theorem 3.13, published version]{B4}.
\cite[Theorem 3.13]{B4} assumes that $S$ is a depiction of $R$, but if $R$ is noetherian, then $S$ will not be a depiction of $R$.  Indeed, in this case the only depiction of $R$ will be itself \cite[Theorem 3.12]{B4}, and $R \not = S$ if $I$ is a non-maximal ideal of $S$.}
\end{Proposition}

\begin{proof}
Since $\operatorname{dim}S/I \geq 1$, $I$ is a non-maximal ideal of $S$.
Thus there is a maximal ideal $\mathfrak{n}$ of $S$ for which $\mathfrak{n} \supset I$.
Since $I$ is a maximal ideal of $R$ and $I \subset \mathfrak{n}$, we have
\begin{equation} \label{awesomeness}
\mathfrak{n} \cap R = I.
\end{equation}

Furthermore, since $I$ is a radical of $S$, there are primes $\mathfrak{p}_1, \ldots, \mathfrak{p}_n$ of $S$ such that $I = \mathfrak{p}_1 \cap \cdots \cap \mathfrak{p}_n$, by the Lasker-Noether theorem.
Fix $h \in \mathfrak{n} \setminus (\mathfrak{p}_1 \cup \cdots \cup \mathfrak{p}_n)$. 
Then for $f \in S$, we have
\begin{equation} \label{Zep}
fh \in I \ \ \ \Rightarrow \ \ \ f \in I.
\end{equation}
Indeed, if $fh \in I$, then $fh \in \mathfrak{p}_i$ for each $i$.
Whence $f \in \mathfrak{p}_i$ for each $i$, and therefore $f \in I$. 

By assumption, $I \not = 0$.
Fix $g \in I \setminus 0$, and consider the chain of ideals of $R$,
$$0 \subset gR \subseteq (g,gh)R \subseteq (g,gh,gh^2)R \subseteq \cdots \subseteq I.$$
We claim that each inclusion is proper.
Indeed, assume to the contrary that there is some $\ell \geq 0$ and $r_0, \ldots, r_{\ell} \in R$ such that
$$gh^{\ell+1} = \sum_{j = 0}^{\ell} r_j gh^j.$$
Then since $S$ is an integral domain,
$$h^{\ell+1} = \sum_{j = 0}^{\ell}r_jh^j.$$
Whence
\begin{equation} \label{hop}
h^{\ell+1} - \sum_{j = 1}^{\ell}r_jh^j = r_0 \in R.
\end{equation}
But $h \in \mathfrak{n}$. 
Therefore $r_0 \in \mathfrak{n} \cap R = I$ by (\ref{awesomeness}).
Furthermore, since $R = k + I$, for each $j \in [0, \ell]$ there is some $\beta_j \in k$ and $t_j \in I$ such that $r_j = \beta_j + t_j$.
Since $r_0$ and each $t_jh^j$ are in $I$, (\ref{hop}) yields
\begin{equation} \label{eternal}
t := h^{\ell+1} - \sum_{j=1}^{\ell}\beta_j h^j = r_0 + \sum_{j=1}^{\ell}t_jh^j \in I \subset \mathfrak{n}.
\end{equation}

The left-hand side implies that $t$ is a polynomial in $k[h]$.
Therefore, since $k$ is algebraically closed, $t$ splits
$$t = h^{\ell+1} - \sum_{j=1}^{\ell}\beta_j h^j = h^m(h - \alpha_1) \cdots (h- \alpha_{\ell-m}),$$
where $m \geq 1$ and $\alpha_1, \ldots, \alpha_{\ell-m} \in k \setminus 0$.
Set $f := (h- \alpha_1) \cdots (h -\alpha_{\ell})$.
By (\ref{eternal}) we have $hf = t \in I$.
Thus, by (\ref{Zep}), $f \in I$.
Consequently, $f \in \mathfrak{n}$.
But this is not possible by Hilbert's Nullstellensatz, since $\alpha_1, \ldots, \alpha_{\ell-m}$ are nonzero scalars, $h$ is in $\mathfrak{n}$, and $\mathfrak{n}$ is a maximal ideal of $S$.
\end{proof}

\begin{Proposition} \label{nn1}
Suppose $\operatorname{dim}S/I_i \geq 1$ for some $i$.
Then $R = \cap_i (k + I_i)$ is nonnoetherian.
\end{Proposition}

\begin{proof}
Suppose $\operatorname{dim}S/I_i \geq 1$.
By Proposition \ref{dim SI}, $I_i$ contains a strict infinite ascending chain of ideals of $k + I_i$,
$$J_1 \subset J_2 \subset J_3 \subset \cdots \subset I_i.$$

(i) We claim that each $J_{\ell}$ is an $R$-module.
Let $r \in R$.
Then $r \in k + I_i$.
Whence $J_{\ell}r \subseteq J_{\ell}$ since $J_{\ell}$ is an ideal of $k + I_i$, proving our claim.

(ii) Now let $a \in \cap_j I_j$.
Then each $aJ_{\ell}$ is in $\cap_j I_j \subset R$.
Thus each $aJ_{\ell}$ is an ideal of $R$ by Claim (i).

Consider the chain of ideals of $R$,
\begin{equation} \label{ac}
aJ_1 \subseteq aJ_2 \subseteq aJ_3 \subseteq \cdots.
\end{equation}
Assume to the contrary that for some $\ell$,
$$aJ_{\ell} = aJ_{\ell+1}.$$
Then for each $b \in J_{\ell+1} \setminus J_{\ell}$, there is some $c \in J_{\ell}$ such that
$$ab = ac.$$
But $S$ is an integral domain.
Whence
$$b = c \in J_{\ell},$$
a contradiction to our choice of $b$.
Thus the chain (\ref{ac}) is strict.
Therefore $R$ is nonnoetherian.
\end{proof}

We recall the following elementary facts.

\begin{Lemma} \label{elementary}
Let $R$ be an integral domain, and let $\mathfrak{p}, \mathfrak{m} \in \operatorname{Spec}R$ be ideals satisfying $\mathfrak{p} \subseteq \mathfrak{m}$.
Then\footnote{We prove Lemma \ref{elementary} for completeness.

(1) It suffices to show that $\mathfrak{p}R_{\mathfrak{m}} \cap R \subseteq \mathfrak{p}$.
Let $\frac ab \in \mathfrak{p}R_{\mathfrak{m}} \cap R$, with $a \in \mathfrak{p}$ and $b \in R \setminus \mathfrak{m}$.
Then
$$b \cdot \frac ab = a \in \mathfrak{p}.$$
Thus, since $b, \frac ab \in R$ and $\mathfrak{p}$ is prime in $R$, we have $b \in \mathfrak{p}$ or $\frac ab \in \mathfrak{p}$.
But $b \not \in \mathfrak{p}$ since $b \not \in \mathfrak{m}$ and $\mathfrak{p} \subseteq \mathfrak{m}$.
Therefore $\frac ab \in \mathfrak{p}$.

(2) Let $\frac{a_1}{b_1}, \frac{a_2}{b_2} \in R_{\mathfrak{m}}$, with $a_1,a_2 \in R$ and $b_1,b_2 \in R \setminus \mathfrak{m}$.
Suppose
$$\frac{a_1}{b_1} \cdot \frac{a_2}{b_2} \in \mathfrak{p} R_{\mathfrak{m}}.$$
We claim that $\frac{a_1}{b_1}$ or $\frac{a_2}{b_2}$ is in $\mathfrak{p}R_{\mathfrak{m}}$.
Indeed, there is some $c \in \mathfrak{p}$ and $d \in R \setminus \mathfrak{m}$ such that
$$\frac{a_1}{b_1} \cdot \frac{a_2}{b_2} = \frac cd.$$
Whence
$$a_1a_2d = b_1b_2c \in \mathfrak{p}.$$
Now $d \not \in \mathfrak{p}$ since $d \not \in \mathfrak{m}$ and $\mathfrak{p} \subseteq \mathfrak{m}$.
Thus $a_1a_2 \in \mathfrak{p}$ since $\mathfrak{p}$ is prime in $R$.
In particular, $a_1 \in \mathfrak{p}$ or $a_2 \in \mathfrak{p}$; say $a_1 \in \mathfrak{p}$.
Then $\frac{a_1}{b_1} \in \mathfrak{p}R_{\mathfrak{m}}$, proving our claim.}
\begin{enumerate}
 \item $\mathfrak{p}R_{\mathfrak{m}} \cap R = \mathfrak{p}$.
 \item $\mathfrak{p}R_{\mathfrak{m}} \in \operatorname{Spec}R_{\mathfrak{m}}$.
\end{enumerate}
\end{Lemma}

Again let $R = \cap_i (k + I_i)$.

\begin{Lemma} \label{double stars}
If $\mathfrak{p} \in \operatorname{Spec}R$ and $\mathfrak{p} \subseteq I_i$ for some $i$, then
$$\mathfrak{p}S \cap R = \mathfrak{p}.$$
\end{Lemma}

\begin{proof}
Suppose the hypotheses hold.
Let $ab \in \mathfrak{p}S \cap R$, with $a \in \mathfrak{p}$ and $b \in S$.
We claim that $ab \in \mathfrak{p}$.
Indeed, by Lemma \ref{Andy} there is some
$$c \in \left( \cap_{j \not = i} I_j \right) \cap R \setminus I_i.$$
Then $ac \in \cap_j I_j$ since $a \in \mathfrak{p} \subseteq I_i$.
Thus for any $s \in S$,
$$acs \in \cap_j I_j \subset R.$$
In particular,
$$acb^2 \in R.$$
Thus, since $a \in \mathfrak{p}$,
$$(ab)^2 \cdot c = a \cdot (acb^2) \in \mathfrak{p}.$$
But $c \in R \setminus \mathfrak{p}$ and $(ab)^2 \in R$.
Thus $(ab)^2 \in \mathfrak{p}$ since $\mathfrak{p}$ is prime in $R$.
Therefore $ab \in \mathfrak{p}$, again since $\mathfrak{p}$ is prime in $R$.
\end{proof}

\begin{Proposition} \label{surjective}
The morphism
$$\iota: \operatorname{Spec}S \to \operatorname{Spec}R, \ \ \ \ \mathfrak{q} \mapsto \mathfrak{q} \cap R,$$
is surjective.
\end{Proposition}

\begin{proof}
Let $\mathfrak{p} \in \operatorname{Spec}R$.
We claim that there is some $\mathfrak{q} \in \operatorname{Spec}S$ such that $\mathfrak{q} \cap R = \mathfrak{p}$.

(i) First suppose $\mathcal{Z}(\mathfrak{p}) \cap U_{S/R} = \emptyset$.
Then there is some $i$ for which $\mathfrak{p} \subseteq I_i$, by Lemma \ref{today}.
Set
$$\mathfrak{t} := \mathfrak{p}(k + I_i)_{I_i} \cap (k + I_i).$$
Recall that $I_i \cap R \in \operatorname{Spec}R$ by Proposition \ref{star}.

(i.a) We have $\mathfrak{p} = \mathfrak{t} \cap R$ since
$$\mathfrak{p} \stackrel{\textsc{(i)}}{=} \mathfrak{p}R_{I_i \cap R} \cap R \stackrel{\textsc{(ii)}}{=} \mathfrak{p}(k+I_i)_{I_i} \cap R = \mathfrak{p}(k+I_i)_{I_i} \cap (k + I_i) \cap R = \mathfrak{t} \cap R,$$
where (\textsc{i}) holds by Lemma \ref{elementary}.1, and (\textsc{ii}) holds by Lemma \ref{external}.

(i.b) We claim that
$$\mathfrak{t} \in \operatorname{Spec}(k+I_i) \ \ \ \text{ and } \ \ \ \mathfrak{t} \subseteq I_i.$$
By Lemma \ref{elementary}.2,
$$\mathfrak{p}R_{I_i \cap R} \in \operatorname{Spec}R_{I_i \cap R}.$$
Thus by Lemma \ref{external},
$$\mathfrak{p}(k + I_i)_{I_i} \in \operatorname{Spec}(k+I_i)_{I_i}.$$
Therefore $\mathfrak{t} \in \operatorname{Spec}(k+I_i)$, since the intersection of a prime ideal with a subalgebra is a prime ideal of the subalgebra.

Furthermore,
$$\mathfrak{t} = \mathfrak{p}(k + I_i)_{I_i} \cap (k+I_i) \subseteq I_i(k+I_i)_{I_i} \cap (k+I_i) \stackrel{\textsc{(i)}}{=} I_i,$$
where (\textsc{i}) holds by Lemma \ref{elementary}.1 since $I_i \in \operatorname{Spec}(k+I_i)$.

(i.c) We claim that
$$\mathfrak{p} = \sqrt[S]{\mathfrak{t}S} \cap R.$$
Indeed,
$$\mathfrak{p} \stackrel{\textsc{(i)}}{=} \mathfrak{t} \cap R \subseteq \sqrt[S]{\mathfrak{t}S} \cap R \stackrel{\textsc{(ii)}}{\subseteq} \sqrt[R]{\mathfrak{t}S \cap R}
= \sqrt[R]{\mathfrak{t}S \cap (k+I_i) \cap R} \stackrel{\textsc{(iii)}}{=}
\sqrt[R]{\mathfrak{t} \cap R} \stackrel{\textsc{(iv)}}{=} \sqrt[R]{\mathfrak{p}} = \mathfrak{p},$$
where (\textsc{i}) and (\textsc{iv}) hold by Claim (i.a);  (\textsc{ii}) holds since if $s^n \in tS$ and $s \in R$, then $s \in \sqrt[R]{tS \cap R}$; and (\textsc{iii}) holds by Claim (i.b) together with Lemma \ref{double stars} (with $k+I_i$ in place of $R$).

(i.d) Since $S$ is noetherian, the Lasker-Noether theorem implies that there are ideals $\mathfrak{q}_1, \ldots, \mathfrak{q}_m \in \operatorname{Spec}S$, minimal over $\sqrt[S]{\mathfrak{t}S}$,
such that
$$\sqrt[S]{\mathfrak{t}S} = \mathfrak{q}_1 \cap \cdots \cap \mathfrak{q}_m.$$
Thus
\begin{equation} \label{p =}
\mathfrak{p} \stackrel{\textsc{(i)}}{=} \sqrt[S]{\mathfrak{t}S} \cap R = \left( \mathfrak{q}_1 \cap \cdots \cap \mathfrak{q}_m \right) \cap R = \left( \mathfrak{q}_1 \cap R \right) \cap \cdots \cap \left( \mathfrak{q}_m \cap R \right),
\end{equation}
where (\textsc{i}) holds by Claim (i.c).
Furthermore, each $\mathfrak{q}_j \cap R$ is a prime ideal of $R$ since $\mathfrak{q}_j \in \operatorname{Spec}S$ and $R \subset S$ (e.g., \cite[Lemma 2.1]{B4}).

Assume to the contrary that for each $j \in [1,m]$,
$$\mathfrak{q}_j \cap R \not = \mathfrak{p}.$$
Then for each $j$ there is some
$$a_j \in (\mathfrak{q}_j \cap R) \setminus \mathfrak{p}.$$
Whence
$$a_1 \cdots a_m \in \cap_j (\mathfrak{q}_j \cap R) \stackrel{\textsc{(i)}}{=} \mathfrak{p},$$
where (\textsc{i}) holds by (\ref{p =}).
But $\mathfrak{p}$ is prime in $R$, a contradiction.
Thus there is some $j$ for which
$$\mathfrak{q}_j \cap R = \mathfrak{p}.$$
Our desired ideal is therefore $\mathfrak{q} := \mathfrak{q}_j \in \operatorname{Spec}S$.

(ii) Now suppose $\mathcal{Z}(\mathfrak{p}) \cap U_{S/R} \not = \emptyset$; say $\mathfrak{n} \in \mathcal{Z}(\mathfrak{p}) \cap U_{S/R}$.
Set
$$\mathfrak{q} := \mathfrak{p}S_{\mathfrak{n}} \cap S.$$
We claim that
$$\mathfrak{q} \cap R = \mathfrak{p} \ \ \ \text{ and } \ \ \ \mathfrak{q} \in \operatorname{Spec}S.$$

First observe that
$$\mathfrak{p} \stackrel{\textsc{(i)}}{=} \mathfrak{p}R_{\mathfrak{n} \cap R} \cap R \stackrel{\textsc{(ii)}}{=} \mathfrak{p}S_{\mathfrak{n}} \cap R = \mathfrak{p}S_{\mathfrak{n}} \cap S \cap R = \mathfrak{q} \cap R,$$
where (\textsc{i}) holds by Lemma \ref{elementary}.1, and (\textsc{ii}) holds since $\mathfrak{n} \in U_{S/R}$.
Furthermore, since $\mathfrak{p} \in \operatorname{Spec}R$, we have $\mathfrak{p}R_{\mathfrak{n} \cap R} \in \operatorname{Spec}(R_{\mathfrak{n} \cap R})$ by Lemma \ref{elementary}.2.
Whence $\mathfrak{p}S_{\mathfrak{n}} \in \operatorname{Spec}S_{\mathfrak{n}}$ since $\mathfrak{n} \in U_{S/R}$.
Therefore $\mathfrak{q} = \mathfrak{p}S_{\mathfrak{n}} \cap S \in \operatorname{Spec}S$.
\end{proof}

\begin{Theorem} \label{main'}
Let $I_1, \ldots, I_n$ be a set of proper non-maximal nonzero radical ideals of $S$ which are pairwise coprime, and set $R := \cap_i (k + I_i)$.
Then
\begin{enumerate}
 \item $R$ is nonnoetherian if and only if there is some $i$ for which $\operatorname{dim}S/I_i \geq 1$.
 \item $R$ is depicted by $S$ if and only if for each $i$, $\operatorname{dim}S/I_i \geq 1$.
\end{enumerate}
\end{Theorem}

\begin{proof}
(1): The implications $\Rightarrow$ and $\Leftarrow$ are respectively Propositions \ref{internal}.1 and \ref{nn1}.

(2): The morphism $\iota: \operatorname{Spec}S \to \operatorname{Spec}R$ is surjective by Proposition \ref{surjective}.
Furthermore, $U_{S/R}$ is nonempty since $U_{S/R} = \left( \cup_i \mathcal{Z}(I_i) \right)^c$ is an open dense subset of $\operatorname{Max}S$, by Proposition \ref{dim SI2}.
It thus suffices to show that
\begin{equation} \label{suffices}
U_{S/R}^c = \cup_i \mathcal{Z}(I_i) \subseteq \left\{ \mathfrak{n} \in \operatorname{Max}S \ | \ R_{\mathfrak{n} \cap R} \text{ is nonnoetherian} \right\},
\end{equation}
where the inclusion holds if and only if $\operatorname{dim}S/I_i \geq 1$ for each $i$.

Suppose $\mathfrak{n} \in \cup_i \mathcal{Z}(I_i)$.
Then $\mathfrak{n}$ contains some $I_j$.
Whence $\mathfrak{n} \cap R = I_j \cap R$ by Proposition \ref{star}.
Thus by Lemma \ref{external},
$$R_{\mathfrak{n} \cap R} = R_{I_j \cap R} = (k + I_j)_{I_j}.$$

$\bullet$ First suppose $\operatorname{dim}S/I_j = 0$.
Then $R_{\mathfrak{n} \cap R} = R_{I_j \cap R}$ is noetherian by Proposition \ref{internal}.2.
Therefore the inclusion in (\ref{suffices}) does not hold.

$\bullet$ Now suppose $\operatorname{dim}S/I_j \geq 1$.
Then $I_j$ contains a strict infinite ascending chain of ideals of $k + I_j$, by Proposition \ref{dim SI}.
Therefore the localization $R_{\mathfrak{n} \cap R} = (k + I_j)_{I_j}$ is nonnoetherian.
In particular, if $\operatorname{dim}S/I_i \geq 1$ for each $i$, then the inclusion in (\ref{suffices}) holds.
\end{proof}

\begin{Corollary}
If $\operatorname{dim}S/I_i \geq 1$ for each $i$, then each of the closed points $I_i \cap R$ of $\operatorname{Spec}R$ has positive geometric dimension.
\end{Corollary}

\begin{proof}
By Theorem \ref{main'}, $S$ is a depiction of $R$.
Therefore for each $i$,
$$\operatorname{gdim}(I_i \cap R) \geq \operatorname{dim}S/I_i \geq 1.$$
\end{proof}

\section{Sheaves of depictions} \label{sheaves}

Let $(X, \mathcal{O})$ be an affine scheme, and set $R := \mathcal{O}(X)$.
We introduce the following definition.

\begin{Definition} \rm{
A \textit{sheaf of depictions} $\tilde{S}$ on $(X, \mathcal{O})$ is a sheaf of algebras such that on each principal open set $D(a) \subset X$, $a \in R$, the algebra $\tilde{S}(D(a))$ is a depiction of $\mathcal{O}(D(a))$.
}\end{Definition}

A sheaf $\mathcal{M}$ on $X$ is said to be a sheaf of modules if, on each open set 
$U \subset X$, $\mathcal{M}(U)$ is an $\mathcal{O}(U)$-module, and for each inclusion of open sets $U \subset V$, the restriction $\mathcal{M}(V) \to \mathcal{M}(U)$ is an $\mathcal{O}(V)$-module homomorphism.
The sheafification of an $R$-module $M$ is the sheaf of modules $\tilde{M}$ defined on each principal open set $D(a)$ by
$$\tilde{M}(D(a)) := M \otimes_{\mathcal{O}(X)} \mathcal{O}(D(a)) = M \otimes_R R[a^{-1}],$$
and on a general open set $U$ by the inverse limit
$$\tilde{M}(U) := \lim_{\substack{\longleftarrow\\ D(a) \subset U}} \tilde{M}(D(a)).$$ 
In this section we show that the sheafification of a depiction is a sheaf of depictions.

Let $S$ be an integral domain and $k$-algebra.
For an element $a \in S$ and ideal $I \subset S$, set $S_a := S[a^{-1}]$ and $I_a := IS[a^{-1}]$.

\begin{Lemma} \label{yup}
Fix $a \in S$.
\begin{enumerate}
 \item If $\mathfrak{q} \in \operatorname{Spec}S$ and $a \not \in \mathfrak{q}$, then $\mathfrak{q}_a \in \operatorname{Spec}S_a$.
 \item If $\mathfrak{n} \in \operatorname{Max}S$ and $a \not \in \mathfrak{n}$, then $\mathfrak{n}_a \in \operatorname{Max}S_a$.
\end{enumerate}
\end{Lemma}

\begin{proof}
(1) Suppose $\mathfrak{q} \in \operatorname{Spec}S$ and $a \not \in \mathfrak{q}$.
Since $S_a$ is a flat $S$-module, the short exact sequence $0 \to \mathfrak{q} \to S \to S/\mathfrak{q} \to 0$ induces the short exact sequence
$$0 \to \mathfrak{q} \otimes_S S_a \to S \otimes_S S_a \cong S_a \to S/\mathfrak{q} \otimes_S S_a \to 0.$$
Whence
\begin{equation} \label{hope}
S/\mathfrak{q} \otimes_S S_a \cong S_a/\mathfrak{q}_a.
\end{equation}
But $S/\mathfrak{q}$ is an integral domain since $\mathfrak{q}$ is prime.
Furthermore, $S/\mathfrak{q} \otimes S_a$ is not the zero ring since $a^n \not \in \mathfrak{q}$ for all $n \geq 0$.
Thus $S/\mathfrak{q} \otimes S_a$ is also an integral domain.
Therefore $\mathfrak{q}_a$ is a prime of $S_a$, by (\ref{hope}).

(2) Suppose $\mathfrak{n} \in \operatorname{Max}S$ and $a \not \in \mathfrak{n}$.
By Claim (1), we have
\begin{equation*} \label{lost}
S/\mathfrak{n} \otimes_S S_a \cong S_a/\mathfrak{n}_a \not = 0.
\end{equation*}
Furthermore, $S/\mathfrak{n} \otimes S_a$ is a field since $\mathfrak{n}$ is a maximal ideal of $S$.
Consequently, $\mathfrak{n}_a$ is a maximal ideal of $S_a$.
\end{proof}

Let $R$ be a subalgebra of $S$.

\begin{Lemma} \label{where}
Fix $a \in R$.
If
$$\iota_{S/R}: \operatorname{Spec}S \to \operatorname{Spec}R$$
is surjective, then so is
$$\iota_{S_a/R_a}: \operatorname{Spec}S_a \to \operatorname{Spec}R_a.$$
\end{Lemma}

\begin{proof}
Suppose $\iota_{S/R}$ is surjective.
Let $\tilde{\mathfrak{p}} \in \operatorname{Spec}R_a$, and set $\mathfrak{p} := \tilde{\mathfrak{p}} \cap R$.
Then $\mathfrak{p}$ is in $\operatorname{Spec}R$.
Thus there is a prime $\mathfrak{q} \in \operatorname{Spec}S$ such that $\mathfrak{q} \cap R = \mathfrak{p}$, by the surjectivity of $\iota_{S/R}$.
Furthermore, the ideal $\mathfrak{q}_a$ is in $\operatorname{Spec}S_a$, by Lemma \ref{yup}.1.

We want to show that $\mathfrak{q}_a \cap R_a = \tilde{\mathfrak{p}}$, from which the lemma follows.

(i) We first claim that $\mathfrak{q}_a \cap R_a \supseteq \tilde{\mathfrak{p}}$.

Let $g \in \tilde{\mathfrak{p}}$.
Then for $\ell \geq 0$ sufficiently large, $a^{\ell}g$ is in $R$.
Whence $a^{\ell}g \in \tilde{\mathfrak{p}} \cap R = \mathfrak{p}$.
Thus $a^{\ell}g \in \mathfrak{q}$.
Therefore $g = a^{-\ell}a^{\ell}g \in \mathfrak{q}_a$.

(ii) We now claim that $\mathfrak{q}_a \cap R_a \subseteq \tilde{\mathfrak{p}}$.

Let $g \in \mathfrak{q}_a \cap R_a$.
Then again for $\ell \geq 0$ sufficiently large, $a^{\ell}g$ is in $\mathfrak{q}$ and $R$.
Thus,
$$a^{\ell}g \in \mathfrak{q} \cap R = \mathfrak{p} = \tilde{\mathfrak{p}} \cap R.$$
Consequently, $g = a^{-\ell}a^{\ell}g \in \tilde{\mathfrak{p}}$.
\end{proof}

\begin{Proposition} \label{what}
Fix $a \in R$.
If $S$ is a depiction of $R$, then $S_a$ is a depiction of $R_a$.
\end{Proposition}

\begin{proof}
Suppose $S$ is a depiction of $R$.

(i) The morphism $\iota_{S_a/R_a}: \operatorname{Spec}S_a \to \operatorname{Spec}R_a$ is surjective by Lemma \ref{where}.

(ii) Let $\mathfrak{n} \in \operatorname{Max}S_a$, and suppose $(R_a)_{\mathfrak{n} \cap R_a}$ is noetherian. 
We claim that 
$$(R_a)_{\mathfrak{n} \cap R_a} = (S_a)_{\mathfrak{n}}.$$

Since $\mathfrak{n}$ is a proper ideal of $S_a$, we have $\mathfrak{n} \not \ni a$.
Therefore
$$(R_a)_{\mathfrak{n} \cap R_a} \stackrel{\textsc{(i)}}{=} R_{\mathfrak{n} \cap R} \stackrel{\textsc{(ii)}}{=} S_{\mathfrak{n} \cap S} \stackrel{\textsc{(iii)}}{=} (S_a)_{\mathfrak{n}},$$
where (\textsc{i}) and (\textsc{iii}) hold since $a \in R \setminus \mathfrak{n}$; and (\textsc{ii}) holds since $R_{\mathfrak{n} \cap R} = (R_a)_{\mathfrak{n} \cap R_a}$ is noetherian and $S$ is a depiction of $R$.

(iii) Finally, we claim that the locus $U_{S_a/R_a}$ is nonempty.

Let $D_S(a) := \{ \mathfrak{n} \in \operatorname{Max}S \ | \ \mathfrak{n} \not \ni a \}$ denote the complement of the vanishing locus of $a$ in $\operatorname{Max}S$.
Then
$$U_{S_a/R_a} = U_{S/R} \cap D_S(a) \not = \emptyset$$
since $U_{S/R}$ and $D_S(a)$ are open dense sets of $\operatorname{Max}S$.
\end{proof}

\begin{Corollary} 
Suppose $S$ is a depiction of $R$.
Then the sheafification $\tilde{S}$ of the $R$-module $S$ on $\operatorname{Spec}R$ is a sheaf of depictions on $\operatorname{Spec}R$. 
\end{Corollary}

\section{Noncommutative blowups of nonnoetherian singularities} \label{nbons}

Let $S$ be a normal integral domain and a finitely generated $k$-algebra.
Let $Y_1, \ldots, Y_n$ be positive dimensional proper subvarieties of $\operatorname{Max}S$ that intersect the smooth locus.
For each $i \in [1,n]$, denote by $I_i := I(Y_i)$ the corresponding radical ideal of $S$.
Consider the nonnoetherian coordinate ring $R := \cap_i (k + I_i)$ and its set of positive dimensional closed points (Proposition \ref{star}),
$$\mathfrak{m}_i := I_i \cap R \in \operatorname{Spec}R.$$

Following \cite[Section R]{L}, we call the endomorphism ring
$$A := \operatorname{End}_R( _RR \oplus \bigoplus_i \mathfrak{m}_i )$$
the `noncommutative blowup' of $\operatorname{Max}R$ at the points $\mathfrak{m}_1, \ldots, \mathfrak{m}_n$.
These points are precisely the nonnoetherian points of $R$ (that is, the points $\mathfrak{m} \in \operatorname{Max}R$ for which $R_{\mathfrak{m}}$ is nonnoetherian), by Theorem \ref{main'} and Proposition \ref{dim SI2}.
Our main theorem in this section is that if either (i) each $Y_i$ is irreducible, or (ii) $n = 1$, then $A$ is a noncommutative desingularization of its center $R$.
Furthermore, $S$ is the cycle algebra of $A$, and thus $A$ provides a means to retrieve $S$ from the knowledge of $R$ alone.
In particular, $R$ is depicted by the cycle algebra of $A$.  

In the following lemma we do not assume $S$ is normal.

\begin{Lemma} \label{ngu}
Let $I$ be a nonzero ideal of a noetherian integral domain $S$, and suppose $I$ is also an ideal of an overring $T \subset \operatorname{Frac}S$ of $S$.
Then $T$ is contained in the integral closure $\overbar{S}$ of $S$.
\end{Lemma}

\begin{proof}
Let $s \in I \setminus \{0\}$ and $t \in T$.
By assumption, $t^{\ell}s \in I$ for each $\ell \geq 0$.
Consider the ascending chain of ideals of $S$
$$sS \subseteq (s,ts)S \subseteq (s, ts, t^2s)S \subseteq (s, ts, t^2s, t^3s) \subseteq \cdots.$$
Since $S$ is noetherian, there is some $m \geq 1$ and $\sigma_0, \ldots, \sigma_{m-1} \in S$ such that
$$t^{m}s = \sum_{j = 0}^{m-1} \sigma_jt^js.$$
Thus, since $S$ is an integral domain and $s \not = 0$, we have
$$t^{m} - \sum_{j=0}^{m-1} \sigma_j t^j = 0.$$
Consequently, $t$ is in the integral closure $\overbar{S}$ of $S$.
\end{proof}

Again let $S$ be a normal finitely generated domain.
For brevity, set
$$R^i := S \cap \left( \cap_{j \not = i} \left(k + I_j \right) \right).$$
We include $S$ in the intersection for the case $n = 1$.

\begin{Lemma} \label{hom mm}
For each $i \in [1,n]$, we have
$$\operatorname{Hom}_R(\mathfrak{m}_i, \mathfrak{m}_i) = \operatorname{Hom}_R(\mathfrak{m}_i, R) = R^i.$$
\end{Lemma}

\begin{proof}
(i) We first claim that $\operatorname{Hom}_R(\mathfrak{m}_i, R) \subseteq S$.

Indeed, $\operatorname{Hom}_S(I_i, I_i)$ is the largest overring of $S$ for which $I_i$ is an ideal.  
Thus, since $S$ is normal, Lemma \ref{ngu} implies
\begin{equation} \label{ngu2}
\operatorname{Hom}_S(I_i,I_i) \subseteq S.
\end{equation}

Let $x \in \operatorname{Hom}_R(\mathfrak{m}_i, R)$ and $w \in I_1I_2 \cdots I_n$.
Then $x^{\ell}w$ is in $\operatorname{Hom}_S(I_i,I_i)$ for each $\ell \geq 1$, since $wI_i \subseteq \mathfrak{m}_i$.
Whence $x^{\ell}w$ is in $S$ by (\ref{ngu2}).
But since $S$ is a normal noetherian domain, the same argument given in the proof of Lemma \ref{ngu}, with $x$ and $w$ in place of $t$ and $s$, shows that $x$ itself is in $S$.

(ii) We now claim that $\operatorname{Hom}_R(\mathfrak{m}_i, R) \subseteq R$.

Consider $x \in \operatorname{Hom}_R(\mathfrak{m}_i, R)$ and $y \in \mathfrak{m}_i$.
Then for each $j \in [1,n]$, $xy$ is in $k+I_j$.
Furthermore, since $y$ is in $R$, there is a $c \in k$ and $z \in I_j$ such that $y = c + z \in k + I_j$.
In particular, $xz$ is in $I_j$, since $x$ is in $S$ by Claim (i).
Thus $x$ itself is in $k+I_j$, since $cx+xz = xy$ is in $k+I_j$.
But $j$ was arbitrary, and therefore $x$ is in $R$.

(iii) Finally, we claim that $R^i \subseteq \operatorname{Hom}_R(\mathfrak{m}_i, \mathfrak{m}_i)$.

Since $\mathfrak{m}_i \subset R \subseteq k + I_j$ for each $j$, and $R^i \subseteq S$, we have $R^i \mathfrak{m}_i \subseteq R$.
Furthermore, $R^i \subseteq S$ implies $R^i\mathfrak{m}_i \subseteq I_i$.
Therefore $R^i\mathfrak{m}_i \subseteq I_i \cap R = \mathfrak{m}_i$.

(iv) We have
$$R^i \stackrel{\textsc{(i)}}{\subseteq} \operatorname{Hom}_R(\mathfrak{m}_i, \mathfrak{m}_i) \subseteq \operatorname{Hom}_R(\mathfrak{m}_i, R) \stackrel{\textsc{(ii)}}{\subseteq} R \subseteq R^i,$$
where (\textsc{i}) holds by Claim (iii), and (\textsc{ii}) holds by Claim (ii).
\end{proof}

\begin{Lemma} \label{onward}
Let $\mathfrak{p}, \mathfrak{q} \in \operatorname{Spec}R$ be coprime ideals.
Then
$$\operatorname{Hom}_R(\mathfrak{p},\mathfrak{q}) = \mathfrak{q}.$$
\end{Lemma}

\begin{proof}
Since $\mathfrak{p},\mathfrak{q}$ are ideals of $R$, $\operatorname{Hom}_R(\mathfrak{p},\mathfrak{q})$ is isomorphic as an $R$-module to the maximum $R$-module $C \subseteq \operatorname{Frac}R$ satisfying $C \mathfrak{p} \subseteq \mathfrak{q}$.
In particular, $C \supseteq \mathfrak{q}$.

To show the reverse inclusion, let $c \in C$.
Since $\mathfrak{p},\mathfrak{q}$ are coprime, there is an $a \in \mathfrak{p}$ and $b \in \mathfrak{q}$ such that $a + b = 1$.
Whence
$$c(1-b) = ca \in C \mathfrak{p} \subseteq \mathfrak{q}.$$
But $\mathfrak{q}$ is prime and $1-b \not \in \mathfrak{q}$.
Thus, $c \in \mathfrak{q}$.
Therefore $C = \mathfrak{q}$.
\end{proof}

\begin{Proposition} \label{tiled}
There is an algebra isomorphism
\begin{equation} \label{matrix}
A = \operatorname{End}_R( _RR \oplus \bigoplus_i \mathfrak{m}_i ) \cong
\begin{bmatrix}
R  & \mathfrak{m}_1 & \mathfrak{m}_2 & \cdots & \mathfrak{m}_n \\
R^1 & R^1 & \mathfrak{m}_2 & \cdots & \mathfrak{m}_n \\
R^2 & \mathfrak{m}_1 & R^2 & \cdots & \mathfrak{m}_n \\
\vdots & \vdots & \vdots & \ddots & \vdots \\
R^n & \mathfrak{m}_1 & \mathfrak{m}_2 & \cdots & R^n
\end{bmatrix}.
\end{equation}
\end{Proposition}

\begin{proof}
Each $\mathfrak{m}_i$ is a prime ideal of $R$, by Proposition \ref{star}.
Furthermore, for each $i \not = j$, there is some
$$a \in I_i \cap R = \mathfrak{m}_i \ \ \ \text{ and } \ \ \ b \in I_j \cap R = \mathfrak{m}_j$$
such that $a + b = 1$, by Lemma \ref{Andy}.
Thus the set of ideals $\mathfrak{m}_1, \ldots, \mathfrak{m}_n$ are pairwise coprime.
The isomorphism (\ref{matrix}) therefore holds by Lemmas \ref{hom mm} and \ref{onward}.
\end{proof}

\begin{Remark} \rm{
The endomorphism ring of the \textit{right} $R$-module $R \oplus \bigoplus_i \mathfrak{m}_i$ is the transpose of the matrix ring given in (\ref{matrix}), and it is not known whether it is cycle regular.
(As a right (resp.\ left) $R$-module, $R \oplus \bigoplus_i \mathfrak{m}_i$ may be viewed as an $n+1$ column (resp.\ row) vector.)
}\end{Remark}

\begin{Remark} \label{n=1} \rm{
In the case $n = 1$, we have $\mathfrak{m} = I$ (omitting the subscript $i$), and the tiled matrix ring (\ref{matrix}) simplifies to
$$A = \operatorname{End}_R( _RR \oplus I ) \cong
\begin{bmatrix}
R & I \\ S & S
\end{bmatrix}.$$
}\end{Remark}

\begin{Proposition} \label{get S}
The cycle algebra of $A$ is $S$.
\end{Proposition}

\begin{proof}
By Proposition \ref{tiled}, the cycle algebra of $A$ is $\tilde{S} := k[R + R^1 + \cdots + R^n]$.
By Remark \ref{n=1}, it suffices to suppose that $n \geq 2$.

We first claim that for any subset $K \subseteq \{ 1, \ldots, n\}$ with $|K| \geq 2$, we have
\begin{equation} \label{=S}
\sum_{i \in K} \bigcap_{j \in K \setminus \{i \} } I_j = S.
\end{equation}
We proceed by induction on $|K|$.
Let $K \ni 1,2$.

First suppose $|K| = 2$.
Then (\ref{=S}) reduces to $I_1 + I_2 = S$, and this holds since $I_1$ and $I_2$ are coprime ideals of $S$.

Now suppose (\ref{=S}) holds for $|K| \leq N$, and let $|K| = N + 1$.
Set $K_1 := K \setminus \{1 \}$ and $K_2 := K \setminus \{2 \}$.
Then
\begin{equation*}
\begin{split}
S & \stackrel{\textsc{(i)}}{=} I_1 + I_2 = I_1 \cap S + I_2 \cap S \\ & \stackrel{\textsc{(ii)}}{=}
I_1 \cap \left( \sum_{i \in K_1} \bigcap_{j \in K_1 \setminus \{i \} } I_j \right) + I_2 \cap \left( \sum_{i \in K_2} \bigcap_{ j \in K_2 \setminus \{i \} } I_j \right) \\
& \subseteq \sum_{i \in K} \bigcap_{j \in K \setminus \{i\}} I_j \subseteq S,
\end{split}
\end{equation*}
where (\textsc{i}) holds since $I_1$ and $I_2$ are coprime, and (\textsc{ii}) holds by induction.
This proves our claim.

Thus,
$$S = \sum_{i = 1}^n \bigcap_{j \not = i} I_j \subseteq \sum_{i = 1}^n R^i \subseteq \tilde{S} \subseteq S.$$
Therefore $\tilde{S} = S$.
\end{proof}

Fix $1 \leq i \leq n$, and let $\mathfrak{q} \in \operatorname{Spec}S$ be a minimal prime over $\mathfrak{m}_i$.
Since $\mathfrak{m}_i$ is a maximal ideal of $R$, we have
\begin{equation} \label{mi =}
\mathfrak{m}_i = I_i \cap R = \mathfrak{q} \cap R.
\end{equation}

\begin{Lemma} \label{q = I}
Suppose $\mathfrak{q} \in \operatorname{Spec}S$ is a minimal prime over $\mathfrak{m}_i$.
Then $I_i \subseteq \mathfrak{q}$.
Consequently, if $I_i$ is prime in $S$, then $I_i = \mathfrak{q}$.
\end{Lemma}

\begin{proof}
We first claim that $I_i \subseteq \mathfrak{q}$.
Let $a \in S \setminus \mathfrak{q}$; we want to show that $a \not \in I_i$.

Assume to the contrary that $a \in I_i$.
By Lemma \ref{Andy}, there is some $b \in R$ that is in $\cap_{j \not = i}I_j \setminus I_i$.
Whence, $ab \in \cap_j I_j \subset R$.
Furthermore, since $b \in R \setminus I_i$, we have $b \not \in \mathfrak{q} \cap R$ by (\ref{mi =}).
In particular, $b \not \in \mathfrak{q}$.
Since $a$ and $b$ are not in $\mathfrak{q}$ and $\mathfrak{q}$ is prime, their product $ab$ is not in $\mathfrak{q}$.
Thus,
$$a^{-1} = b(ab)^{-1} \in R_{\mathfrak{q} \cap R} \stackrel{(\textsc{i})}{=} (k+I_i)_{I_i},$$
where (\textsc{i}) holds by Lemma \ref{external}.
Whence $a^{-1} \in (k+I_i)_{I_i}$.
But $a \in I_i$, and thus $a$ is not invertible in $(k+I_i)_{I_i}$, a contradiction.
Therefore $I_i \subseteq \mathfrak{q}$.
\end{proof}

\begin{Lemma} \label{upward}
For each minimal prime $\mathfrak{q} \in \operatorname{Spec}S$ over $\mathfrak{m}_i$ and $j \not = i$, the following hold:
$$R_{\mathfrak{q} \cap R } = R^j_{\mathfrak{q} \cap R^j} = \mathfrak{m}_j (k+I_i)_{I_i} = (k+I_i)_{I_i} \ \ \ \ \text{ and } \ \ \ \ \mathfrak{m}_j S_{\mathfrak{q}} = S_{\mathfrak{q}}.$$
Furthermore, if either $I_i$ is prime in $S$ or $n = 1$, then
\begin{equation*} \label{n1}
R^i_{\mathfrak{q} \cap R^i} = S_{\mathfrak{q}} \ \ \ \ \text{ and } \ \ \ \ \mathfrak{m}_iS_{\mathfrak{q}} = \mathfrak{q}S_{\mathfrak{q}}.
\end{equation*}
\end{Lemma}

\begin{proof}
(i) By Lemma \ref{external}, we have $R_{\mathfrak{q} \cap R } = (k+I_i)_{I_i}$, and for $j \not = i$, $R^j_{\mathfrak{q} \cap R^j} = (k+I_i)_{I_i}$.

(ii) Let $j \not = i$.
We claim that $\mathfrak{m}_j(k+I_i)_{I_i} = (k+I_i)_{I_i}$.
Fix $b \in \mathfrak{m}_j \setminus I_i$.
Then $b \in k+I_i$ since $b \in R$.
Whence, $b^{-1} \in (k+I_i)_{I_i}$.
Therefore
$$1 = bb^{-1} \in \mathfrak{m}_j (k+I_i)_{I_i}.$$

(iii) Let $j \not = i$.
We claim that $\mathfrak{m}_jS_{\mathfrak{q}} = S_{\mathfrak{q}}$.
By Lemma \ref{Andy}, there is some
$$b \in (I_j \cap R) \setminus I_i = \mathfrak{m}_j \setminus \mathfrak{m}_i.$$
Whence, $b \not \in \mathfrak{q}$ by (\ref{mi =}).
Therefore
$$1 = bb^{-1} \in \mathfrak{m}_j S_{\mathfrak{q}}.$$

(iv) Suppose $I_i$ is prime in $S$.
We claim that $R^i_{\mathfrak{q}\cap R^i} = S_{\mathfrak{q}}$.
Clearly, $R^i_{\mathfrak{q} \cap R^i} \subseteq S_{\mathfrak{q}}$.

To show the reverse inclusion, suppose $\frac ab \in S_{\mathfrak{q}}$ with $a \in S$ and $b \in S \setminus \mathfrak{q}$.
By Lemma \ref{Andy}, there is some $c \in \left( \cap_{j \not = i} I_j \right) \setminus I_i$.
Thus, $ac$ and $bc$ are in $\cap_{j \not = i}I_j \subset R^i$.
Furthermore, $c \not \in \mathfrak{q}$ since $\mathfrak{q} = I_i$, by Lemma \ref{q = I}.
Whence $bc \not \in \mathfrak{q}$ since $\mathfrak{q}$ is prime.
Therefore
$$\frac ab = \frac{ac}{bc} \in R^i_{\mathfrak{q} \cap R^i},$$
proving our claim.

(v) Again suppose $I_i$ is prime in $S$.
We claim that $\mathfrak{m}_iS_{\mathfrak{q}} = \mathfrak{q}S_{\mathfrak{q}}$.
Clearly, $\mathfrak{m}_iS_{\mathfrak{q}} \subseteq \mathfrak{q}S_{\mathfrak{q}}$.

To show the reverse inclusion, let $a \in \mathfrak{q} = I_i$.
Fix $b \in \cap_{j \not = i}I_j \setminus I_i$.
Then
$$ab \in \cap_j I_j \subset R.$$
Whence, $ab \in I_i \cap R = \mathfrak{m}_i$.
Furthermore, $b \in S \setminus \mathfrak{q}$ since $\mathfrak{q} = I_i$.
Therefore
$$a = abb^{-1} \in \mathfrak{m}_iS_{\mathfrak{q}}.$$

(vi) Finally, suppose $n = 1$, in which case $\mathfrak{m} = I$ (we omit the subscript $i$).
We claim that $IS_{\mathfrak{q}} = \mathfrak{q}S_{\mathfrak{q}}$.
The inclusion $IS_{\mathfrak{q}} \subseteq \mathfrak{q}S_{\mathfrak{q}}$ follows from Lemma \ref{q = I}.

To show the reverse inclusion, let $a \in \mathfrak{q}$.
Consider the set of minimal primes over $I$,
$$\mathfrak{q}_1 := \mathfrak{q}, \mathfrak{q}_2, \ldots, \mathfrak{q}_m \in \operatorname{Spec}S.$$
In particular, $I = \cap_j \mathfrak{q}_j$ since $I$ is radical.

For each $2 \leq j \leq m$, fix $b_j \in \mathfrak{q}_j \setminus \mathfrak{q}$.
Then $b_2 \cdots b_m \in S \setminus \mathfrak{q}$ since $\mathfrak{q}$ is prime.
Therefore
$$a = (ab_2\cdots b_m )(b_2 \cdots b_m)^{-1} \in \left(\cap_j \mathfrak{q}_j \right) S_{\mathfrak{q}} = I S_{\mathfrak{q}}.$$
\end{proof}

Set
$$\tilde{R} := (k+I_i)_{I_i} + \mathfrak{q}S_{\mathfrak{q}}.$$
If $I_i$ is prime in $S$, then by Lemma \ref{q = I} this reduces to
$$\tilde{R} = (k+ \mathfrak{q})_{\mathfrak{q}} + \mathfrak{q}S_{\mathfrak{q}}.$$

\begin{Lemma} \label{wonder}
$\tilde{R}$ is a subalgebra of $S_{\mathfrak{q}}$.
\end{Lemma}

\begin{proof}
Since $k+I_i \subset S$, it suffices to show that if $a$ is invertible in $(k+I_i)_{I_i}$, then $a$ is also invertible in $S_{\mathfrak{q}}$.
So suppose $a \in (k+I_i) \setminus I_i$.
Then $a = c + \alpha$, where $c \in k^{\times}$ and $\alpha \in I_i$.
But $I_i \subseteq \mathfrak{q}$ by Lemma \ref{q = I}.
Whence $a \in S \setminus \mathfrak{q}$.
\end{proof}

Index the rows and columns of $A_{\mathfrak{q}}$ by $0,1, \ldots, n$.
Denote by $e_{ij} \in M_{n+1}(\operatorname{Frac}S)$ the matrix with a $1$ in the $ij$-th slot and zeros elsewhere, and set $e_i := e_{ii}$.

\begin{Proposition} \label{guesshow}
Suppose each $I_i \subset S$ is prime, or $n = 1$.
Fix $\mathfrak{p} \in \operatorname{Spec}R$, and let $\mathfrak{q} \in \operatorname{Spec}S$ be a minimal prime over $\mathfrak{p}$.
\begin{enumerate}
 \item If $R_{\mathfrak{p}}$ is noetherian, then the cyclic localization $A_{\mathfrak{q}}$ at $\mathfrak{q}$ is the full matrix ring
$$A_{\mathfrak{q}} = M_{n+1}(R_{\mathfrak{p}}) \cong A \otimes_R R_{\mathfrak{p}}.$$
 \item If $R_{\mathfrak{p}}$ is nonnoetherian, then $\mathfrak{p} = \mathfrak{m}_i$ for some $i$, and
\begin{equation*} \label{Aq}
A_{\mathfrak{q}} = \begin{bmatrix}
\tilde{R} & \cdots & \tilde{R} & \mathfrak{q}S_{\mathfrak{q}} & \tilde{R} & \cdots & \tilde{R} \\
\vdots & & \vdots & \vdots & \vdots & & \vdots \\
\tilde{R} & \cdots & \tilde{R} & \mathfrak{q}S_{\mathfrak{q}} & \tilde{R} & \cdots & \tilde{R} \\
S_{\mathfrak{q}} & \cdots & S_{\mathfrak{q}} & S_{\mathfrak{q}} & S_{\mathfrak{q}} & \cdots & S_{\mathfrak{q}}\\
\tilde{R} & \cdots & \tilde{R} & \mathfrak{q}S_{\mathfrak{q}} & \tilde{R} & \cdots & \tilde{R} \\
\vdots & & \vdots & \vdots & \vdots & & \vdots \\
\tilde{R} & \cdots & \tilde{R} & \mathfrak{q}S_{\mathfrak{q}} & \tilde{R} & \cdots & \tilde{R}
\end{bmatrix} \subset M_{n+1}(\operatorname{Frac}S)
\end{equation*}
where the $i$th row and column are respectively
\begin{equation*}
\begin{split}
e_iA_{\mathfrak{q}} & = [ \begin{array}{ccc} S_{\mathfrak{q}} & \cdots & S_{\mathfrak{q}} \end{array} ] = S_{\mathfrak{q}}^{\oplus n +1},\\
A_{\mathfrak{q}}e_i & = [ \begin{array}{ccccccc} \mathfrak{q}S_{\mathfrak{q}} & \cdots & \mathfrak{q}S_{\mathfrak{q}} & S_{\mathfrak{q}} & \mathfrak{q}S_{\mathfrak{q}} & \cdots & \mathfrak{q}S_{\mathfrak{q}} \end{array} ]^{\operatorname{t}},
\end{split}
\end{equation*}
and all other entries are $\tilde{R}$.
\end{enumerate}
\end{Proposition}

\begin{proof}
(1) Suppose $R_{\mathfrak{p}}$ is noetherian.
Then $R_{\mathfrak{p}} = S_{\mathfrak{q}}$ since $S$ is a depiction of $R$. 

(1.i) We first claim that the diagonal entries of $A_{\mathfrak{q}}$ are all $S_{\mathfrak{q}}$.
Fix $i \in [1,n]$.
We have 
$$S_{\mathfrak{q}} \stackrel{\textsc{(i)}}{=} R_{\mathfrak{q} \cap R} \stackrel{\textsc{(ii)}}{\subseteq} R^i_{\mathfrak{q} \cap R^i} \stackrel{\textsc{(iii)}}{\subseteq} S_{\mathfrak{q}},$$
where (\textsc{i}) holds since $S$ is a depiction of $R$; (\textsc{ii}) holds since $R \subset R^i$; and (\textsc{iii}) holds since $R^i \subseteq S$.
Therefore $R^i_{\mathfrak{q} \cap R^i} = S_{\mathfrak{q}}$.

(1.ii) We now claim that the off-diagonal entries of $A_{\mathfrak{q}}$ are also all $S_{\mathfrak{q}}$.

Fix $i \in [1,n]$, and assume to the contrary that $\mathfrak{m}_i \subseteq \mathfrak{q}$.
Then 
$$\mathfrak{m}_i = \mathfrak{m}_i \cap R \subseteq \mathfrak{q} \cap R = \mathfrak{p}.$$
Whence, $\mathfrak{m}_i = \mathfrak{p}$ since $\mathfrak{m}_i$ is maximal. 
But $R_{\mathfrak{m}_i}$ is nonnoetherian by Theorem \ref{main'} and Proposition \ref{dim SI2}, contrary to our choice of $\mathfrak{p}$.
Thus $\mathfrak{m}_i \not \subseteq \mathfrak{q}$.
Hence, there is some $a \in \mathfrak{m}_i \setminus \mathfrak{q}$.
Consequently, $1 = aa^{-1} \in \mathfrak{m}_i S_{\mathfrak{q}}$.
Therefore $\mathfrak{m}_iS_{\mathfrak{q}} = S_{\mathfrak{q}}$.
Together with (1.i), this implies that the off-diagonal entries of $A_{\mathfrak{q}}$ in columns $1, \ldots, n$ are $S_{\mathfrak{q}}$.

Finally, the off-diagonal entries in column $0$ are also $S_{\mathfrak{q}}$: since $1 \in R^i$ and $R^i \subseteq S$, we have $R^iS_{\mathfrak{q}} = S_{\mathfrak{q}}$.

(2) Follows from Lemmas \ref{upward} and \ref{wonder}.
\end{proof}

Fix $i \in [1,n]$ and a minimal prime $\mathfrak{q} \in \operatorname{Spec}S$ over $\mathfrak{m}_i$.

\begin{Lemma} \label{kernel}
Let $j \in [0,n]$, let $P$ be a projective $A_{\mathfrak{q}}$-module, and let
$$\delta: A_{\mathfrak{q}}e_j \to P$$
be an $A_{\mathfrak{q}}$-module homomorphism.
Suppose $e_{ij} \in A_{\mathfrak{q}}$.
If $\delta(e_{ij}) = 0$, then $\delta \equiv 0$.
\end{Lemma}

\begin{proof}
Set $\Lambda := A_{\mathfrak{q}}$, and suppose $\delta(e_{ij}) = 0$.
Let $\ell \geq 1$ be minimal such that $P$ is a direct summand of $\Lambda^{\oplus \ell}$.
Let $a_1, \ldots, a_{\ell} \in \Lambda$ be such that
$$\delta(e_j) = (a_1, \ldots, a_{\ell}) \in \Lambda^{\oplus \ell}.$$

Each $a_k$ is in $e_j\Lambda$ since
$$(a_1, \ldots, a_{\ell}) = \delta(e_j) = \delta(e_j^2) = e_j \delta(e_j) \in e_j\Lambda^{\oplus \ell}.$$
Furthermore, each product $e_{ij}a_k$ is zero since
$$(e_{ij}a_1, \ldots, e_{ij}a_{\ell}) = e_{ij}(a_1, \ldots, a_{\ell}) = e_{ij} \delta(e_j) = \delta(e_{ij}) = 0.$$
But $e_{ij} \alpha \not = 0$ for all nonzero $\alpha$ in $e_j\Lambda$. 
Therefore each $a_k$ is zero.
\end{proof}

\begin{Proposition} \label{leq dim}
The left global dimension of $A_{\mathfrak{q}}$ is bounded above by the Krull dimension of $S_{\mathfrak{q}}$,
$$\operatorname{gldim}A_{\mathfrak{q}} \leq \operatorname{dim}S_{\mathfrak{q}}.$$
\end{Proposition}

\begin{proof}
Set $\Lambda := A_{\mathfrak{q}}$ and $d := \operatorname{dim}S_{\mathfrak{q}} - 1$.
Let $V$ be a $\Lambda$-module.
We claim that
\begin{equation} \label{d+1}
\operatorname{pd}_{\Lambda}(V) \leq d+1.
\end{equation}
It suffices to show that there is a projective resolution $P_{\bullet}$ of $V$,
\begin{equation*} \label{resolution}
\cdots \longrightarrow P_2 \stackrel{\delta_2}{\longrightarrow} P_1 \stackrel{\delta_1}{\longrightarrow} P_0 \stackrel{\delta_0}{\longrightarrow} V \to 0,
\end{equation*}
for which the $(d+1)$th syzygy $\operatorname{ker}\delta_d$ is a projective $\Lambda$-module \cite[Proposition 8.6.iv]{R}.

Since $\{e_0, \ldots, e_{n}\}$ is a complete set of orthogonal idempotents of $\Lambda$, we may assume that for each $\ell \geq 0$ and $j \in [0,n]$, there is some $m_{\ell j} \geq 0$ such that
\begin{equation*} \label{Pl}
P_{\ell} = \bigoplus_{j \, : \, m_{\ell j} \geq 1}(\Lambda e_j)^{\oplus m_{\ell j}} = \bigoplus_{j \, : \, m_{\ell j} \geq 1} \, \bigoplus_{t \in [1,m_{\ell j}]} \Lambda e_j \varepsilon_t, 
\end{equation*}
where $e_j \varepsilon_t$ generates the $t$-th $\Lambda e_j$ summand of $(\Lambda e_j )^{\oplus m_{\ell j}}$ over $\Lambda$.

Now $e_i\Lambda$ is a left $S_{\mathfrak{q}}$-module since $e_i\Lambda e_i = e_iS_{\mathfrak{q}}$.
Furthermore, $e_i\Lambda$ is a projective, hence flat, right $\Lambda$-module.
Thus, setting $\otimes := \otimes_{\Lambda}$, the sequence of $S_{\mathfrak{q}}$-modules
$$\ldots \longrightarrow e_i\Lambda \otimes P_2 \stackrel{1 \otimes \delta_2}{\longrightarrow} e_i\Lambda \otimes P_1 \stackrel{1 \otimes \delta_1}{\longrightarrow} e_i\Lambda \otimes P_0 \stackrel{1 \otimes \delta_0}{\longrightarrow} e_i\Lambda \otimes V \to 0$$
is exact.
Moreover, each term $e_i\Lambda \otimes P_{\ell}$ is a free $S_{\mathfrak{q}}$-module since
\begin{equation} \label{rad dude}
e_i\Lambda \otimes P_{\ell} = \bigoplus_{j \, : \, m_{\ell j} \geq 1} (e_i\Lambda \otimes \Lambda e_j)^{\oplus m_{\ell j}} \cong \bigoplus_j (e_i\Lambda e_j)^{\oplus m_{\ell j}} = \bigoplus_j (e_{ij} S_{\mathfrak{q}})^{\oplus m_{\ell j}}.
\end{equation}
It follows that $e_i\Lambda \otimes P_{\bullet}$ is a free resolution of the $S_{\mathfrak{q}}$-module $e_i\Lambda \otimes V \cong e_iV$.
Thus, since $S_{\mathfrak{q}}$ is a regular local ring of dimension $d+1$, the $(d+1)$th syzygy $\operatorname{ker}(1 \otimes \delta_d)$ of $e_i\Lambda \otimes P_{\bullet}$ is a free $S_{\mathfrak{q}}$-module.
Therefore, since $\operatorname{ker}(1 \otimes \delta_d)$ is a submodule of $\bigoplus_j e_{ij}S_{\mathfrak{q}}^{\oplus m_{d j}}$, for each $j \in [0,n]$ there is some $r_j \in [0, m_{d j}]$ such that
\begin{equation} \label{wow}
\operatorname{ker}(1 \otimes \delta_d) \cong \bigoplus_{j \, : \, r_j \geq 1} (e_{ij}S_{\mathfrak{q}})^{\oplus r_j}.
\end{equation}

Again since $e_i\Lambda$ is a flat right $\Lambda$-module, the sequence
$$0 \longrightarrow e_i\Lambda \otimes \operatorname{ker}\delta_d \longrightarrow e_i\Lambda \otimes P_d \stackrel{1 \otimes \delta_d}{\longrightarrow} e_i\Lambda \otimes P_{d-1}$$
is exact.
Whence
\begin{equation} \label{wow2}
e_i\Lambda \otimes \operatorname{ker}\delta_d = \operatorname{ker}(1 \otimes \delta_d).
\end{equation}
But (\ref{wow}) and (\ref{wow2}) together imply
\begin{equation} \label{life}
e_i \operatorname{ker}\delta_d = \bigoplus_{j \, : \, r_j \geq 1} (e_{ij}S_{\mathfrak{q}})^{\oplus r_j} = e_i \bigoplus_j (\Lambda e_j)^{\oplus r_j} = e_i \bigoplus_{j \, : \, r_j \geq 1} \, \bigoplus_{t \in [1,r_j]} \Lambda e_j \varepsilon_t.
\end{equation}
In particular, for each $t \in [1,r_j]$ we have $\delta_d(e_{ij} \varepsilon_{t}) = 0$.
Thus by Lemma \ref{kernel},
$$\delta_d(\Lambda e_j\varepsilon_t) = 0.$$
Therefore
\begin{equation} \label{a}
\operatorname{ker}\delta_d \supseteq \bigoplus_{j \, : \, r_j \geq 1} \, \bigoplus_{t \in [1,r_j]} \Lambda e_j \varepsilon_t.
\end{equation}

To show the reverse inclusion, fix $j \in [0,n]$ satisfying $m_{d j} \geq 1$, and let $t \in [1,m_{d j}]$.
Suppose $e_{kj}\varepsilon_t \in \operatorname{ker}\delta_d$.
Then, since $1 \in \Lambda^{ik}$ for each $k \in [0,n]$, we have
$$\delta_d(e_{ij}\varepsilon_t) = \delta_d(e_{ik}e_{kj} \varepsilon_t) = e_{ik}\delta_d(e_{kj}\varepsilon_t) = 0.$$
Whence $e_{ij}\varepsilon_t \in e_i \operatorname{ker} \delta_d$. 
Thus $t \in [1,r_j]$, by (\ref{life}).
Therefore
\begin{equation} \label{b}
\operatorname{ker}\delta_d \subseteq \bigoplus_{j \, : \, r_j \geq 1} \, \bigoplus_{t \in [1,r_j]} \Lambda e_j \varepsilon_t.
\end{equation}
(\ref{a}) and (\ref{b}) together imply that $\operatorname{ker}\delta_d = \bigoplus (\Lambda e_j)^{\oplus r_j}$.
Consequently, (\ref{d+1}) holds.
\end{proof}

\begin{Proposition} \label{simple prop}
The cyclic localization $A_{\mathfrak{q}}$ has precisely two simple modules up to isomorphism,
\begin{equation} \label{simples}
\begin{split}
V & := A_{\mathfrak{q}}e_i/A_{\mathfrak{q}}(1-e_i)A_{\mathfrak{q}}e_i \cong S_{\mathfrak{q}}/\mathfrak{q}, \\
W & := A_{\mathfrak{q}}e_0/A_{\mathfrak{q}}e_iA_{\mathfrak{q}}e_0 \cong \bigoplus_{j \not = i} ke_{j0} \cong (\tilde{R}/\mathfrak{q})^{\oplus n}.
\end{split}
\end{equation}
Their projective dimensions are respectively
$$\operatorname{pd}_{A_{\mathfrak{q}}}(V) = \operatorname{dim}S_{\mathfrak{q}} \ \ \ \text{ and } \ \ \ \operatorname{pd}_{A_{\mathfrak{q}}}(W) = 1.$$
\end{Proposition}

\begin{proof}
Set $\Lambda := A_{\mathfrak{q}}$.

(i) We claim that the simple $\Lambda$-modules are precisely the modules $V$ and $W$ in (\ref{simples}).
For each $j \in [0,n] \setminus \{i\}$, there is a (left) $\Lambda$-module isomorphism
$$\Lambda e_0 \stackrel{\cdot e_{0j}}{\longrightarrow} \Lambda e_j.$$
Furthermore, $W$ is simple since for each $j,k \in [0,n] \setminus \{i \}$, the matrix entry $A^{jk}$ contains $1 \in R$; whence 
$$e_{jk}e_{k0} = e_{j0} \ \ \ \text{ and } \ \ \ e_{kj}e_{j0} = e_{k0}.$$

(ii) We claim that $\operatorname{pd}_{\Lambda}(V) = \operatorname{dim}S_{\mathfrak{q}}$.
Indeed, we have
$$\operatorname{dim}S_{\mathfrak{q}} \stackrel{\textsc{(i)}}{=} \operatorname{pd}_{S_{\mathfrak{q}}}(S_{\mathfrak{q}}/\mathfrak{q}) \stackrel{\textsc{(ii)}}{\leq} \operatorname{pd}_{\Lambda}(V) \stackrel{\textsc{(iii)}}{\leq} \operatorname{dim}S_{\mathfrak{q}},$$
where (\textsc{i}) holds since $S_{\mathfrak{q}}$ is a regular local ring, and (\textsc{iii}) holds by Proposition \ref{leq dim}.
(\textsc{ii}) holds since if $P_{\bullet}$ is a projective resolution of $V$ over $\Lambda$, then $e_i\Lambda \otimes_{\Lambda} P_{\bullet}$ is a free resolution of $V \cong S_{\mathfrak{q}}/\mathfrak{q}$ over $e_i\Lambda e_i \cong S_{\mathfrak{q}}$, as shown in (\ref{rad dude}).

(iii) We claim that $\operatorname{pd}_{\Lambda}(W) = 1$.
Consider the complex
\begin{equation} \label{complex}
0 \to \Lambda e_i \stackrel{\cdot e_{i0}}{\longrightarrow} \Lambda e_0 \longrightarrow W \to 0.
\end{equation}
The module homomorphism $\Lambda e_i \stackrel{\cdot e_{i0}}{\longrightarrow} \Lambda e_0$ maps onto the kernel of $\Lambda e_0 \to W$, namely $\Lambda e_i \Lambda e_0$, since $\Lambda^{ii} = S_{\mathfrak{q}} = \Lambda^{i0}$.
Thus the complex (\ref{complex}) is exact.
\end{proof}

\begin{Proposition} \label{localized}
The residue module at $\mathfrak{q}$ decomposes as
\begin{equation} \label{residue module}
A_{\mathfrak{q}}/\mathfrak{q} = V \oplus W^{\oplus n},
\end{equation}
and has projective dimension
\begin{equation*} \label{pd of residue module}
\operatorname{pd}_{A_{\mathfrak{q}}}(A_{\mathfrak{q}}/\mathfrak{q}) = \operatorname{dim} S_{\mathfrak{q}}.
\end{equation*}
\end{Proposition}

\begin{proof}
Set $\Lambda := A_{\mathfrak{q}}$.

The direct sum decomposition (\ref{residue module}) follows from Proposition \ref{simple prop}, where we view $V$ and $W$ as columns of the $(n+1) \times (n+1)$ tiled matrix ring $\Lambda/(\mathfrak{q} \cap \Lambda)$.

The projective dimension of $\Lambda/(\mathfrak{q} \cap \Lambda)$ equals the Krull dimension of $S_{\mathfrak{q}}$ since $$\operatorname{dim}S_{\mathfrak{q}} \stackrel{\textsc{(i)}}{=} \operatorname{pd}_{\Lambda}(V) \stackrel{\textsc{(ii)}}{\leq} \operatorname{pd}_{\Lambda}(\Lambda/(\mathfrak{q} \cap \Lambda)) \leq \operatorname{gldim}\Lambda \stackrel{\textsc{(iii)}}{\leq} \operatorname{dim}S_{\mathfrak{q}}.$$
Indeed, (\textsc{i}) holds by Proposition \ref{simple prop}; (\textsc{ii}) holds by (\ref{residue module}), since the projective dimension of a module $M$ is greater than or equal to the projective dimension of any direct summand of $M$; and (\textsc{iii}) holds by Proposition \ref{leq dim}.
\end{proof}

Let $A$ be a $k$-algebra with prime center $Z$, and let $\mathfrak{m} \in \operatorname{Max}Z$.
Then $A_{\mathfrak{m}} = A \otimes_Z Z_{\mathfrak{m}}$ is said to be Azumaya over its center $Z_{\mathfrak{m}}$ if $A_{\mathfrak{m}}$ is a free $Z_{\mathfrak{m}}$-module of finite rank, and the algebra homomorphism
\begin{equation*} 
\begin{array}{ccc}
A_{\mathfrak{m}} \otimes_{Z_{\mathfrak{m}}} A_{\mathfrak{m}}^{\operatorname{op}} & \to & \operatorname{End}_{Z_{\mathfrak{m}}}(A_{\mathfrak{m}})\\
a \otimes b & \mapsto & (x \mapsto axb)
\end{array}
\end{equation*}
is an isomorphism \cite[13.7.6]{McR}, \cite[III.1.3]{BG}.
The Azumaya locus of $A$ is the set of points $\mathfrak{m} \in \operatorname{Max}Z$ for which $A_{\mathfrak{m}}$ is an Azumaya algebra.  

\begin{Remark} \label{ily} \rm{
It is well-known that if $A_{\mathfrak{m}}$ is free of finite rank over $Z_{\mathfrak{m}}$, then $A_{\mathfrak{m}}$ is Azumaya if and only if $A_{\mathfrak{m}}/\mathfrak{m}A_{\mathfrak{m}}$ is a central simple algebra over $k$, if and only if $A_{\mathfrak{m}}/\mathfrak{m}A_{\mathfrak{m}} \cong M_n(k)$ for some $n \geq 1$ (assuming $k$ is algebraically closed).
}\end{Remark}

\begin{Theorem} \label{akzily}
Let $S$ be a finite type normal integral domain. 
Let $I_1, \ldots, I_n$ be a set of proper non-maximal nonzero radical ideals of $S$ such that either $n = 1$, or the closed sets $\mathcal{Z}(I_i) \subset \operatorname{Max}S$ are irreducible and pairwise non-intersecting.

Set $R := \cap_{i=1}^n (k + I_i)$, and consider its nonnoetherian points $\mathfrak{m}_i := I_i \cap R \in \operatorname{Max}R$. 
Then
\begin{enumerate}
 \item $S$ can be retrieved from $R$ as the cycle algebra of the endomorphism ring 
$$A = \operatorname{End}_R\left(R \oplus \mathfrak{m}_1 \oplus \cdots \oplus \mathfrak{m}_n \right).$$ 
Furthermore, the center of $A$ is $R$.
 \item The Azumaya locus of $A$ and the noetherian locus $U_{S/R}$ of $R$ coincide.
 \item If each $\mathcal{Z}(I_i)$ intersects the smooth locus of $\operatorname{Max}S$,  then $A$ is a noncommutative desingularization of its center $R$: 
\begin{enumerate}
 \item $\operatorname{Frac}R$ and $A \otimes_R \operatorname{Frac}R$ are Morita equivalent, and
 \item for each $i \in [1,n]$ and minimal prime $\mathfrak{q} \in \operatorname{Spec}S$ over $\mathfrak{m}_i$, we have
$$\operatorname{gldim}A_{\mathfrak{q}} = \operatorname{pd}_{A_{\mathfrak{q}}}(A_{\mathfrak{q}}/\mathfrak{q}) = \operatorname{dim}S_{\mathfrak{q}}.$$
\end{enumerate} 
\end{enumerate}
\end{Theorem}

\begin{proof}
(1) The algebra $A$ has center $R$ by Proposition \ref{tiled}, and cycle algebra $S$ by Proposition \ref{get S}.

(2) The noetherian locus $U_{S/R}$ of $R$ is contained in the Azumaya locus of $A$ by Proposition \ref{guesshow}.1 and Remark \ref{ily}.
Conversely, if $\mathfrak{n} \in \operatorname{Max}S \setminus U_{S/R}$, then $A \otimes_R R_{\mathfrak{n} \cap R}$ is not a free $R_{\mathfrak{n} \cap R}$-module by Proposition \ref{tiled}.
Whence $\mathfrak{n} \cap R \in \operatorname{Max}R$ is not in the Azumaya locus of $A$.

(3.a) We claim that $\operatorname{Frac}R$ and $A \otimes_R \operatorname{Frac}R$ are Morita equivalent.
By Theorem \ref{main'}, $S$ is a depiction of $R$; in particular, $U_{S/R} \not = \emptyset$.
Thus $\operatorname{Frac}R = \operatorname{Frac}S$.
Therefore
$$A \otimes_R \operatorname{Frac}R = A \otimes_R \operatorname{Frac}S \stackrel{\textsc{(i)}}{=} M_{n+1}(\operatorname{Frac}S) = M_{n+1}(\operatorname{Frac}R),$$
where (\textsc{i}) holds by Proposition \ref{tiled}.
The claim follows.

(3.b) We have $\operatorname{gldim}A_{\mathfrak{q}} \leq \operatorname{dim}S_{\mathfrak{q}}$ by Proposition \ref{leq dim}, and $\operatorname{gldim}A_{\mathfrak{q}} \geq \operatorname{dim}S_{\mathfrak{q}}$ by Proposition \ref{simple prop}.
Furthermore, $\operatorname{pd}_{A_{\mathfrak{q}}}(A_{\mathfrak{q}}/\mathfrak{q}) = \operatorname{dim}S_{\mathfrak{q}}$ by Proposition \ref{localized}.
\end{proof}

\begin{Remark} \rm{
The advantage of the noncommutative blowup $A$ over the depiction $S$ is given by Theorem \ref{akzily}.2: the noetherian locus $U_{S/R}$ of $R$ is intrinsic to $A$ since it is encoded in the representation theory of $A$, whereas the noetherian locus is invisible to $S$. 
Furthermore, $A$ `sees' both $R$ and $S$: they appear as the center and cycle algebra of $A$ respectively. 
}\end{Remark}

\ \\
\textbf{Acknowledgments.}
The author was supported by the Austrian Science Fund (FWF) grant P 30549-N26.
Part of this article is based on work supported by Heilbronn Institute for Mathematical Research.

\bibliographystyle{hep}
\def\cprime{$'$} \def\cprime{$'$}

\end{document}